\documentclass[10pt,a4paper]{article}
\usepackage{articleSetting}
\usepackage[square,authoryear]{natbib}
\usepackage{palatino}
\usepackage{hyperref}

\begin{document}
\title{BSDEs, c\`adl\`ag martingale problems and orthogonalisation
 under basis risk.}
\author{Ismail LAACHIR$^{(1, 2)}$ and Francesco RUSSO$^{(1)}$}
\date{March 18th 2016}
\maketitle
\vspace{10mm}
{\bf Abstract.}
\\
The aim of this paper is to introduce a new formalism for 
the deterministic analysis  associated
with backward stochastic differential equations 
 driven by general c\`adl\`ag martingales.
When the martingale is a standard Brownian motion, 
the natural deterministic analysis
is provided by the solution of a semilinear PDE of parabolic type.
A significant application concerns 
 the hedging problem under basis risk
of a contingent claim $g(X_T,S_T)$,
 where $S$ 
(resp. $X$) is an underlying price of a traded 
(resp. non-traded but observable)
asset, via the celebrated F\"ollmer-Schweizer decomposition.
We revisit the case when the couple of price processes $(X,S)$ is
a diffusion and we 
provide explicit expressions when $(X,S)$ is 
an exponential of additive processes.

\begin{footnotesize}
\begin{itemize}
\item[(1)] ENSTA ParisTech, Universit\'e Paris-Saclay,
 Unit\'e de Math\'ematiques
Appliqu\'ees, F-91120 Palaiseau, France.
 
\item[(2)] Universit\'e de Bretagne Sud, Lab-STICC, F-56321 Lorient Cedex,
 France.
\end{itemize}
\end{footnotesize}



MSC Classification 2010:   91G80; 60G51; 60J25; 60J75; 60H10;  60H30; 91G80
\bigskip

\section{Introduction}

The motivation of this work comes from the hedging problem  in the presence of basis risk.
  When a derivative product  is based on a non traded or illiquid underlying,
 the specification of a hedging strategy becomes problematic. 
In practice one  frequent methodology consists 
 in constituting a portfolio based on a (traded and liquid) additional asset  
which is correlated with the  original one. 
The use of a non perfectly correlated asset induces a residual risk, often called \textbf{basis risk}, that makes the market incomplete. 
A common example is the hedging of a basket (or index based) option, only using  a subset of the 
assets composing the contract. Commodity markets also present many situations where basis risk plays an essential role, since many goods (as kerosene)
  do not have liquid future markets. For instance, kerosene consumers as airline companies, who want to hedge
 their exposure to the fuel use alternative future contracts, as crude oil or  heating oil. 
 The latter two commodities are strongly correlated to kerosene and their
 corresponding  future market 
is liquid. 
Weather derivatives constitute an example of contract written
 on a non-traded underlying,
 since they are based on heating temperature; natural gas or electricity are in general used to hedge these contracts.

In this work, we consider a maturity $T > 0$, a pair of processes $(X,S)$
and a contingent claim of the type $h:=g(X_T)$ or even  $h:=g(X_T,S_T)$.
  $X$ is a non traded  or illiquid, but observable  asset   and 
$S$ is a traded asset, 
correlated to $X$. 
In order to hedge this derivative, in general the practitioners
use the \textit{proxy} asset $S$ as a hedging instrument, but since 
the two assets are not perfectly correlated, a basis risk exists.
Because of the incompleteness of this market, one should define a risk aversion criterion. One possibility is to use the utility function based approach to define the hedging strategy, see for example \cite{Davis2006}, \cite{Henderson2002}, \cite{Monoyios2004}, \cite{Monoyios2007}, \cite{CeciGirardi1, CeciGirardi2}, \cite{Ankirchner2010NonTradable}. We mention also \cite{Ankirchner2013BasisRiskLiq} who consider the case when an investor has two possibilities, either hedge with an illiquid instrument, which implies liquidity costs, or hedge using a liquid correlated asset, which entails basis risk.
Another  approach is based on the quadratic hedging error criterion: it follows the idea of the seminal work of \cite{FS1991} that introduces the theoretical bases
 of the quadratic hedging in incomplete markets. In particular, they show the close relation of the quadratic hedging problem with a special semimartingale decomposition, known 
as the F\"ollmer-Schweizer (F-S) decomposition. 
The reader can consult \cite{FS1994, Schweizer2001} for basic information on F-S decomposition,
 which provides the so called {\it local risk minimizing hedging strategy} and it is a significant tool
for solving the {\it mean variance hedging problem} in an incomplete market.

\cite{Hulley2008} applied this general framework to the quadratic hedging under basis risk in a simple log-normal model. 
They consider for instance the two-dimensional Black-Scholes model 
for the non traded (but observable)  $X$  and the hedging asset 
$S$, described by 
\begin{eqnarray*}
d X_t &=& \mu_X X_t dt + \sigma_X X_t dW^X_t, \\
d S_t &=& \mu_S S_t dt + \sigma_S S_t dW^S_t,
\end{eqnarray*}
where $(W^X,W^S)$ is a standard correlated two-dimensional Brownian motion.
They derive the F-S decomposition of a European payoff 
$h = g(X_T)$, i.e. 
\begin{equation}
\label{introFS}
g(X_T) = h_0 + \int_0^T Z^h_u dS_u + L^h_T,
\end{equation}
where $L^h$ is a martingale which is strongly orthogonal to the martingale part of the hedging asset process $S$. Using the Feynman-Kac theorem, they relate
 the decomposition components $h_0$ and $Z^h$ to a PDE terminal-value problem.
This yields, as byproduct, the price and hedging portfolio of 
the European option $h$.
 These quantities can be expressed in closed formulae
 in the case of call-put options. Extensions of those results
to the case of stochastic correlation between the two assets $X$ and $S$,
have been performed by \cite{Ankirchner2012StochCorrel}.

 Coming back to the general case, the F-S decomposition of $h$ with respect to  the $\cF_t$-semimartingale $S$ 
can be seen as a special case of the well-known backward stochastic differential equations (BSDEs). We look for a triplet  of processes $(Y,Z,O)$ 
being solution of an equation of the form
\begin{equation}
\label{BSDEI}
Y_t = h + \int_t^T \hat{f}(\w, u, Y_{u-}, Z_u)d V^S_u - \int_t^T Z_u dM^S_u - (O_T-O_t),
\end{equation}
where $M^S$ (resp. $V^S$) is the local martingale (resp. the bounded variation process) appearing in the semimartingale decomposition of $S$, 
 $O$ is a strongly orthogonal martingale to $M^S$,
and $ \hat f(\omega,s, y, z) = - z$.  

 BSDEs were first studied in the Brownian framework by \cite{Pardoux1990adapted} with an early paper of 
\cite{Bismut73}.
\cite{Pardoux1990adapted} showed existence and uniqueness of
 the solutions when the coefficient $ \hat f$ is globally Lipschitz
with respect to  $(y, z)$ and $h$ being square integrable. 
It was followed by a long series of contributions, see for example
 \cite{ElKarouiSurveyBSDE} for a survey on Brownian based BSDEs
and applications to finance. 
For example, the Lipschitz condition was
essential in $z$ and only a monotonicity condition is required for $y$. 
Many other generalizations were considered.
We also drive the attention on the recent monograph by
 \cite{pardoux2014stochastic}.

 When the driving martingale in the BSDE is a Brownian motion, $h = g(S_T)$,
and $S$ is a Markov diffusion,
a solution of a BSDE  constitutes a probabilistic representation
of a semilinear parabolic PDE. In particular if $u$ is a solution of
 the mentioned PDE, then, roughly speaking 
setting $Y_t = u(t,S_t), Z = \partial_s u (t,S_t), O \equiv 0$, 
the triplet $(Y,Z,O)$ is a solution to \eqref{BSDEI}. 
 That PDE is a deterministic problem naturally related to the BSDE.
When $M^S$  is a general c\`adl\`ag martingale, the link 
between a BSDE \eqref{BSDEI} and a deterministic problem is less obvious.

As far as backward stochastic differential equations driven by a martingale, the first paper seems to be \cite{Buckdahn93}.
 Later, several authors have contributed to that subject, for instance \cite{Briand2002robustness} and  \cite{ElKarouiHuang1997}. 
More recently \cite[Theorem 3.1]{CarFerSan08} give sufficient conditions for existence and  uniqueness for BSDEs of the form \eqref{BSDEI}. BSDEs with partial information driven by c\`adl\`ag martingales were investigated by \cite{CCR1, CCR2}.

In this paper we consider a forward-backward SDE, issued from \eqref{BSDEI}, 
where the forward process solves a sort of martingale problem (in the strong probability sense, i.e. where
the underlying filtration is fixed)
instead of the usual stochastic differential equation appearing in the Brownian case.
More particularly we suppose the existence of an adapted 
continuous bounded variation process $A$,
of an operator $\a: \Da \subset \cC([0,T] \times \R^2)   \rightarrow 
\cL$, where $\cL$ is a suitable space of functions $[0,T] \times \R^2 \rightarrow \C$ (see \eqref{E23}),
such that $(X,S)$ verifies 
\begin{equation*}
\forall y \in  \Da,\quad \left(y(t,X_t,S_t) - \int_0^t \a(y)(u,X_{u-},S_{u-}) d A_u\right)_{0\leq t\leq T} 
\quad \text{is an $\cF_t$-local martingale.}
\end{equation*}
With $\a$ we associate the operator $\widetilde \a$ defined by
$$ \widetilde{\a}(y):= \a(\widetilde{y}) - y \a(id) - id \a(y), \quad
id(t,x,s) = s,  \widetilde y =  y \times id.$$ 
In the forward-backward BSDE we are interested in,  
 the driver $\hat f$ verifies 
\begin{equation} \label{EFormTildef} 
 \a (id)(t, X_{t-}(\omega), S_{t-}(\omega)) 
\hat f(\omega, t, y, z)     =  f(t,  X_{t-}(\omega), S_{t-}(\omega), y, z), 
\ (t,y,z) \in [0,T] \times \C^2, \omega \in \Omega,
\end{equation}
for some $f: [0,T] \times \R^2 \times \C^2 \rightarrow \C$.
 The main idea is to settle a deterministic problem 
which is naturally associated with the forward-backward SDE 
 \eqref{BSDEI}.

The deterministic problem consists in looking for a pair of functions $(y,z)$ 
which solves
\begin{align}
\label{E134}
\begin{split}
\a(y)(t, x, s) &= - f(t, x, s, y(t, x, s), z(t, x, s))   \\
\widetilde{\a}(y)(t, x, s)&=z(t, x, s) \widetilde{\a}(id)(t, x, s), 
\end{split}
\end{align}
for all $t\in [0,T]$ and $(x,s)\in\R^2$, with the terminal condition $y(T,.,.) = g(.,.)$.

Any solution to the deterministic problem \eqref{E134}
 will provide a solution $(Y,Z,O)$ to the corresponding BSDE,
setting $ Y_t = y(t, X_t, S_t), Z_t = z(t, X_{t-}, S_{t-}).$

For illustration, let us consider the elementary case when
 $S$ is a diffusion process fulfilling 
$ dS_t = \sigma_S(t,S_t) dW_t +  b_S (t,S_t) dt,$
and $X \equiv 0$. Then
$A_t \equiv t$, $\langle M^S\rangle = \int_0^\cdot (\sigma_S)^2(u,S_u)du$, $V^S = \int_0^\cdot b_S(u,S_u) du 
=\int_0^\cdot \a(id)(u,S_u) du;   $
  $ \a$ is the parabolic generator of $S$, $\Da = C^{1,2}([0,T] \times \R^2) \rightarrow \C$. In that case  \eqref{E134} becomes
\begin{align}
\label{E134Diff}
\begin{split}
\partial_t y(t, x, s) + (b_S \partial_s y + \frac{1}{2} \sigma_S^2 \partial_{ss} y) (t,x,s) 
 &= - f(t, x, s, y(t, x, s), z(t, x, s))  \\
z   &= \partial_s y 
\end{split}
\end{align}
In that situation  $\widetilde \a$ is closely related to the classical  derivation operator.
When $S$ models the price of a traded asset and
 $f(t,x,s,y,z) = -  b_S (t,s) z$, 
the resolution of \eqref{E134Diff} emerging from  the BSDE \eqref{BSDEI} with   \eqref{EFormTildef},
allows to solve the usual (complete market Black-Scholes type)  hedging problem 
with underlying $S$.
Consequently, in the general case, $\widetilde \a$ appears to be naturally 
associated with a  sort of ''generalized derivation map''.
A first link between the hedging problem in incomplete markets and generalized
 derivation operators
was observed for instance in \cite{gorSAA}.

The aim of our paper is threefold.
\begin{enumerate}[label=\arabic*)]
\item To provide  a general methodology for solving forward-backward SDEs
driven by a c\`adl\`ag martingale, via the solution of a deterministic problem generalizing
the classical partial differential problem appearing in the case of Brownian martingales.
\item To give applications to the hedging problem in the case of
 basis risk via the F\"ollmer-Schweizer decomposition. In particular
 we revisit the case when $(X,S)$ is a diffusion process
whose particular case of  Black-Scholes was treated by
  \cite{Hulley2008}, discussing some analysis related to a corresponding PDE. 
\item To furnish a quasi-explicit solution when the pair of processes $(X,S)$ 
is an exponential of additive processes, which constitutes a generalization of the results
of \cite{gor2013variance} and \cite{Hubalek2006}, established in the absence of basis risk.
 This yields a characterization of the hedging strategy in terms of 
Fourier-Laplace transform and the 
moment generating function.
\end{enumerate}
The paper is organized as follows. In Section \ref{S2}, we state 
the strong inhomogeneous martingale problem, 
 and we give several examples, as
 Markov flows and
 the exponential of additive processes. 
In Section  \ref{S3_BSDE}, we state the general form of a  BSDE driven by a martingale
and we associate  a deterministic problem with it.
 We show in particular that a solution for this deterministic problem yields a solution for the BSDE. 
In Section \ref{S4_BSDE}, we apply previous methodology to the F-S decomposition problem  under basis risk.
 In the case of exponential of additives processes, we 
obtain a quasi-explicit decomposition of the mentioned F-S decomposition.

\section{Strong inhomogeneous martingale problem}
\label{S2}

\setcounter{equation}{0}

\subsection{General considerations}

In this paper $T$ will be a strictly positive number. We consider a complete probability space $(\Omega, \cF, \P)$
 with a filtration $(\cF_t)_{t\in[0,T]}$, fulfilling the usual conditions.
By default, all the processes will be indexed by $[0,T]$.
Let $(X,S)$  a couple of $\cF_t$-adapted 
processes. We will often mention concepts as {\it martingale, semimartingale, adapted, predictable}
without mentioning the underlying filtration $(\cF_t)_{t\in[0,T]}$.
Given a bounded variation function $\phi:[0,T] \rightarrow \R$,
we will denote by $t \mapsto \Vert \phi \Vert_t$ the associated total
variation function.

We introduce a notion  of  martingale type problem related to $(X,S)$, which is a generalization of 
a stochastic differential equation.
We emphasize that the present notion looks similar to the classical notion of \cite{StroockVaradhanBook} but here the notion is 
probabilistically strong and relies on a fixed filtered probability space.
In the context of Stroock and Varadhan, however,  a solution is a
 probability measure and the underlying process is the canonical process on 
some canonical space. Here a filtered probability space is given at
the beginning. A similar notion was introduced in \cite{trutnau}
Definition 5.1. 
A priori, we will not suppose that our
strong martingale problem is well-posed (existence and uniqueness).

\begin{definition}
\label{DSMP}
Let $\cO$ be an open set of $\R^2$. Let $(A_t)$ be an $\cF_t$-adapted
  bounded variation continuous process, such that a.s. $
dA_t \ll d\rho_t,$ for some bounded variation function $\rho$, and $\a$ a map
\begin{equation} \label{D26}
\a : \Da\subset \mathcal{C}([0,T]\times \cO, \C) \longrightarrow \mathcal{L},
\end{equation} 
where 
\begin{align}
\label{E23}
\begin{split}
\mathcal{L} = \lbrace f :& [0,T]\times \cO \rightarrow \C, \text{such that for every compact $K$ of $\cO$ } \\
& \norm{f}_K(t):=\sup_{(x,s)\in K}|f(t,x,s)| <\infty \quad d\rho_t\; a.e.\rbrace.
\end{split}
\end{align}
We say that a couple of c\`adl\`ag processes $(X,S)$ is a solution of the 
{\bf strong martingale problem} related to $(\Da, \a, A)$ , 
if for any $g\in \Da$, $(g(t,X_t,S_t))$ is a  semimartingale with
 continuous bounded variation component such that
\begin{equation}
\label{aFinite} 
\int_{0}^{t} |\a(g)(u,X_{u-}, S_{u-})| d\norm{A}_u  <\infty \; a.s.
\end{equation}
and
\begin{equation}
\label{SMgDecomp} 
t \longmapsto g(t,X_t, S_t) - \int_{0}^{t} \a(g)(u,X_{u-}, S_{u-}) dA_u 
\end{equation} is an $\cF_t$- local martingale.
\end{definition}

We start introducing some significant notations. 
\begin{notation} \label{NSpecial}\
\begin{enumerate} [label=\arabic*)]
\item  $id: (t,x,s)\longmapsto s$.
\item For any $y \in \mathcal{C}([0,T]\times \cO)$, we denote by $\widetilde{y}$ the function $\widetilde{y}:=y\times id$,  i.e. 
\begin{equation} \label{EidS}
(y\times id)(t,x,s) = s y(t,x,s).
\end{equation}
\item Suppose that $id\in \Da$. For $y \in \Da$ such that $\widetilde{y}\in\Da$, we set
\begin{equation}
\label{EAtilde}
\widetilde{\a}(y):= \a(\widetilde{y}) - y \a(id) - id \a(y).
\end{equation}
\end{enumerate}
 
\end{notation}
As we have mentioned in the introduction, the map $\widetilde{\a}$ will play
 the role of a generalized derivative. We state first a preliminary lemma.
\begin{lemma}\label{LC}
Let $(X,S)$ be a solution of the strong martingale problem related to $(\Da, \a, A)$ (as in Definition \ref{DSMP}). Let $y$ be a function such that $y, id, y\times id \in \Da$. We set $Y = y(\cdot,X_\cdot,S_\cdot)$ and $M^Y$ be its martingale component given in \eqref{SMgDecomp}. Then
$
\langle M^Y, M^S \rangle = \int_0^\cdot \widetilde\a(y)(u,X_{u-},S_{u-})dA_u.
$

\end{lemma}
\begin{proof}
In order to compute the angle bracket  $\langle M^Y, M^S \rangle$, we
 start by expressing the corresponding square bracket.
First, notice that, since $y, id \in \Da$ and $A$ is a continuous process, then the bounded variation parts of the semimartingales $S$ and $y(\cdot,X_\cdot, S_\cdot)$ are continuous. We have, for $t\in[0,T]$,
$$
[M^Y,M^S]_t = [Y,S]_t= (SY)_{t} - \int_0^t Y_{u-}dS_u - \int_0^t S_{u-} dY_u,
$$
where the first equality is justified by the fact that the square bracket of any process with a
 continuous bounded variation process vanishes. Using moreover 
the fact that $y\times id \in \Da$, the process
\begin{multline*}
[M^Y,M^S] - \int_0^\cdot \a(y\times id)(u,X_{u-},S_{u-})dA_u + 
\int_0^\cdot y(u,X_{u-},S_{u-}) \a(id)(u,X_{u-},S_{u-})dA_u  
\\ + \int_0^\cdot S_{u-}  \a(y)(u,X_{u-},S_{u-})dA_u
\end{multline*} 
is an $\cF_t$-local martingale.
Consequently, $[M^Y,M^S]$ is a special $\cF_t$-semimartingale, because the integrals above with respect to $A$ are predictable.
Finally, since $\langle M^Y, M^S\rangle - [M^Y,M^S]$ is a local martingale, the uniqueness of the canonical decomposition of the special semimartingale $[M^Y,M^S]$ yields the desired result.

\end{proof}
In the sequel, we will make the following assumption.
\begin{assumption}
\label{E0}
 $(\Da, \a, A)$ verifies the following axioms.
i)  $id \in \Da$. ii) $ (t,x,s) \mapsto s^2 \in \Da$.
\end{assumption}

\begin{corollary} \label{RSpecial} 
Let  $(X,S)$ be  a solution of the strong
 martingale problem introduced in  Definition \ref{DSMP} then,
under Assumption \ref{E0}, $S$ is a special semimartingale with
decomposition $M^S + V^S$ given below.
\begin{enumerate}[label=\roman*)]
\item \label{RSpecial_1} $V^S_t= \int_0^t \a (id)(u, X_{u-},S_{u-}) dA_u$.
\item \label{RSpecial_2} $\langle M^S \rangle_t = 
\int_0^t {\widetilde \a} (id)(u, X_{u-},S_{u-}) dA_u$.
\end{enumerate}
\end{corollary}
\begin{proof}
\ref{RSpecial_1} is obvious since $id \in \Da$ and \ref{RSpecial_2}
is a consequence of 
 Lemma \ref{LC} and the fact that $(t,x,s) \mapsto s^2$
belongs to $\Da$.
\end{proof}

In many situations, the operator $\a$ is related to the classical infinitesimal generator, when it exists. We will 
make this relation explicit in the below example  of Markov processes.

\subsection{The case of Markov semigroup}
\label{SMarkov}

In this section we only consider a single process $S$ instead of a couple $(X,S)$. Without restriction of
generality $\cO$ will be chosen to be $\R$.
Here  $(\cF_t)$ will indicate the canonical filtration associated with $S$.
For this reason,  it is more comfortable to re-express Definition \ref{DSMP} 
into the following simplified version. 


\begin{definition}
\label{DSMPMod}
We say that $S$ is a solution of the 
{\bf strong martingale problem} related to $(\Da, \a, A)$
 with $A_t \equiv  t$, 
if there is a map
\begin{equation}
\label{pbMg}
\a : \Da\subset \mathcal{C}([0,T]\times \R) \longrightarrow \cL,
\end{equation} 
where 
\begin{align}
\label{ELX}
\begin{split}
\mathcal{L} = \lbrace f:&[0,T]\times \R \rightarrow \R, \text{such that for every compact $K$ of $\R$ } \\
& \norm{f}_K(t):=\sup_{s\in K}|f(t,s)| <\infty \quad dt \; a.e.\rbrace,
\end{split}
\end{align}
such that for any $g\in \Da$, $g(t,S_t)$ is a (special)  $\cF_t$-semimartingale with continuous bounded variation component
verifying
\begin{equation}
\label{aFiniteMarkov} 
\int_{0}^{T} |\a(g)(u,S_{u-})| du  <\infty \; a.s.
\end{equation}
and
\begin{equation}
\label{SMgDecompMarkov} 
t \longmapsto g(t,S_t) - \int_{0}^{t} \a(g)(u,S_{u-}) du 
\end{equation} is an $\cF_t$- local martingale.
\end{definition} 
Let $(X^{u,x}_t)_{t\geq u, x\in \R}$ be a time-homogeneous Markovian flow. In particular, if $S=X^{0,x}$ and   $f$ is a bounded Borel function,
 then 
\begin{equation}
\label{MarkovDef}
\E{f(S_t) | \cF_u} =  \E{f(X^{0,y}_{t-u})}\vert_{y = S_u},
\end{equation}
where $0 \leq u \leq t \leq T$.
We suppose moreover that $X_t^{0,x}$ is square integrable
for any $0 \le t \le T$ and $x \in \R$.
We denote by $E$ the linear space of functions such that
\begin{equation}
\label{setE}
E = \Big\{ f\in \cC \; \text{such that }\; \widetilde{f}:= s\mapsto \frac{f(s)}{1+s^2}\; \text{is uniformly continuous and bounded}\;\Big\},
\end{equation}
equipped with the  norm $$\norm{f}_E:= \sup_s{\frac{|f(s)|}{1+s^2}} < \infty.$$

The set $E$ can easily be shown to be a Banach space equipped with the norm $\norm{.}_E$. 
Indeed $E$ is a suitable space for Markov processes which are square integrable. In particular, \eqref{MarkovDef} 
remains valid if $f \in E$.
From now on we consider the family of linear operators $(P_t, t \ge 0)$ defined on the space $E$ by
\begin{equation}
\label{EPT}
P_t f(x) = \E{f(X_t^{0,x})}, \text{for } t\in[0,T], x\in \R, \quad \forall f \in E.
\end{equation}

We formulate now a fundamental assumption.
\begin{assumption}\
\label{A_Markov}
\begin{enumerate}[label=\roman*)]
\item \label{EStable} $P_t E \subset E$ for all $t\in[0,T]$.
\item \label{PBounded} The linear operator $P_t$ is bounded, for all $t\in[0,T]$.
\item \label{strongCont}  $(P_t)$ is {\bf strongly continuous}, i.e. $\displaystyle\lim_{t \rightarrow 0} P_t f = f$ in the $E$ topology.
\end{enumerate}

\end{assumption}

 Using the Markov flow property \eqref{MarkovDef},
 it is easy to see that the family of 
 continuous operators $(P_t)$ defined above has  the semigroup property. 
In particular, under Assumption \ref{A_Markov}, the family $(P_t)$ is strongly continuous semigroup on $E$.
Assumption \ref{A_Markov} is fulfilled in many common cases, as mentioned in Proposition \ref{PropMarkov}
and Remarks \ref{R25} and \ref{RStrongCont}.

The proposition below concerns the validity of items  \ref{EStable} and \ref{PBounded}. 
\begin{proposition} \label{PropMarkov}
Let $t\in[0,T]$.
Suppose that $x\mapsto X^{0,x}_t$ is differentiable in $L^2(\Omega)$ such that
\begin{equation}
\label{derFlow}
\sup_{x\in\R}\E{\vert\partial_x X^{0,x}_t\vert^2} < \infty.
\end{equation}
Then  $P_t f \in E$ for all $f\in E$ and  $P_t$ is a bounded operator.
\end{proposition}
The proof of this proposition is reported in Appendix \ref{PropMarkovProof}.

\begin{remark} \label{R25}
Condition \eqref{derFlow} of Proposition \ref{PropMarkov} is fulfilled in the following two cases.
\begin{enumerate}[label=\arabic*)] 
\item \label{R25i} If $\Lambda$ is a L\'evy process, the Markov flow 
$X^{0,x} = x + \Lambda$ verifies $\partial_x X^{0,x} = 1$.
\item \label{R25ii}
 If $(X^{0,x}_t)$ is a diffusion process verifying
$$
X^{0,x}_t = x + \int_0^t b(X^{0,x}_u)du + \int_0^t \sigma(X^{0,x}_u) dW_u,
$$ where $b$ and $\sigma$ are $C^1_b$ functions.
\end{enumerate}
\end{remark}

\begin{remark} \label{RStrongCont}
Item \ref{strongCont} of Assumption \ref{A_Markov} is verified in the case of square integrable L\'evy processes, 
cf. Proposition \ref{LevyPtStongCont} in Appendix \ref{appA_BSDE}.
\end{remark}

For the rest of this subsection we work under Assumption \ref{A_Markov}.

 Item \ref{strongCont} of Assumption \ref{A_Markov} permits to introduce
  the definition of the 
 generator of $(P_t)$ as follows.
\begin{definition}\label{DGen}
The {\bf generator} $L$ of $(P_t)$ in $E$ is defined on the domain $D(L)$ which is the subspace of E defined by
\begin{equation} \label{MarkovGen}
D(L) = \Big\{ f\in E \; \text{such that }\; \lim_{t \to 0} \dfrac{P_t f-f}{t} \; \text{exists in} \;E\;\Big\}.
\end{equation}
We denote by $Lf$ the limit above.
We  refer to  \cite[Chapter 4]{JacobBookVol1},  
for more details.
\end{definition}

\begin{remark} \label{R100} 
If $f\in E$ such that there is $g\in E$ such that 
$$ (P_t f)(x) - f(x) - \int_0^t P_ug(x) du = 0, \; \forall t\geq 0, \;x\in E,$$ 
then $f\in D(L)$ and $g = Lf$. 
Previous integral is always defined as $E$-valued Bochner integral. 
Indeed, since $(P_t)$ is strongly continuous, then by \cite[Lemma 4.1.7]{JacobBookVol1}, we have  
\begin{equation} \label{LStable}
\Vert P_t \Vert  \le M_w e^{wt},
\end{equation}
 for some real $w$
and related constant $M_w.$
$\Vert \cdot \Vert$ denotes here the operator norm.
\end{remark}



A useful result which allows to deal with time-dependent functions is given below.
\begin{lemma}
\label{lemmaMarkov}
Let $f: [0,T] \to D(L) \subset E$. We suppose the following properties to be verified.
\begin{enumerate}[label=\roman*)]
  \item \label{lemMA1} $f$ is continuous as a $D(L)$-valued function, where $D(L)$ is equipped with the graph norm.
  \item \label{lemMA2} $f:[0,T] \to E$ is of class $C^1$.
\end{enumerate} Then, the below $E$-valued equality holds:
\begin{equation}
\label{lemmaSGroup}
P_t f(t,.) = f(0,.) + \int_{0}^{t} P_u(Lf(u,.)) du + \int_{0}^{t} P_u(\frac{\partial f}{\partial u}(u,.)) du , \; \forall t \in [0,T].
\end{equation} 
\end{lemma} 
\begin{remark} \label{RStable}
We observe that the two integrals above can be considered
 as $E$-valued Bochner integrals 
because, by hypothesis, $t \mapsto Lf(t,\cdot)$ and $ t \mapsto \frac{\partial f}{\partial t}(t,.)$
are continuous with values in $E$, and so 
we can apply again \eqref{LStable} in Remark \ref{R100}.

\end{remark}

\begin{proof} 
It will be enough to show that
\begin{equation} \label{ElemmaSGroup} 
 \frac{d}{dt} P_t f(t,.) = P_t(Lf(t,.)) + P_t \left(\frac{\partial f}{\partial t}(t,.) \right), \; \forall t \in [0,T].
\end{equation}
 In fact, even if Banach space valued, a differentiable function at each point is also absolutely continuous.

Since the right-hand side of \eqref{ElemmaSGroup} is continuous
it is enough to show that the right-derivative of $t \mapsto P_tf(t,\cdot)$
coincides with the right-hand side of \eqref{ElemmaSGroup}.
Let $h>0$. We evaluate
$ P_{t+h}f(t+h,.) - P_tf(t,.) = I_1(t,h) + I_2(t,h), $
where 
$ I_1(t,h) = P_{t+h}f(t+h,.) - P_tf(t+h,.), I_2(t,h) = 
P_{t}f(t+h,.) - P_tf(t,.).$
Now by \cite[Lemma 4.1.14]{JacobBookVol1}, we get
$ I_1(t,h) := P_{t+h}f(t+h,.) - P_tf(t+h,.) = \int_{t}^{t+h} P_u Lf(t+h, .)du.$
We divide by $h$ and we get 
\begin{eqnarray*}
\norm{ \frac{1}{h} \int_{t}^{t+h} (P_u Lf(t+h, .) - P_u Lf(t, .))du}_E
 &\leq&  \frac{1}{h} \int_{t}^{t+h} du \norm{P_u\left\{ Lf(t+h, .)-Lf(t, .)\right\}}_E  \\
&\leq& \norm{Lf(t+h, .)-Lf(t, .)}_E \frac{1}{h} \int_{t}^{t+h} \norm{P_u}du   \\
&\leq& \norm{f(t+h, .)-f(t, .)}_{D(L)} \frac{1}{h} \int_{t}^{t+h} \norm{P_u}du,  
\end{eqnarray*}where $\|.\|_{D(L)}$ is the graph norm of $L$. This converges to zero (notice that $\norm{P_u}$ is bounded by \eqref{LStable}),
 and we get that
$
\frac{1}{h}I_1(t,h) \xrightarrow{h\to 0}  P_t(Lf(t,.)).
$ We estimate now $I_2(t,h)$. 
$$
\norm{\dfrac{P_tf(t+h,.) - P_tf(t,.)}{h} -P_t(\frac{\partial f}{\partial t}(t,.)) }_{E} \leq 
\norm{P_t} \norm{\dfrac{f(t+h,.) - f(t,.)}{h} -\frac{\partial f}{\partial t}(t,.) }_{E}. 
$$ This goes to zero as $h$ goes to zero, by Assumption \ref{lemMA2}.
This concludes the proof of Lemma \ref{lemmaMarkov}.

\end{proof}

We can now discuss the fact that a process $S = X^{0,x}$, where $(X^{u,x}_t)_{t\geq u, x\in \R}$ is a Markovian
flow (as precised at the beginning of Section \ref{SMarkov})
 is a solution to our (time inhomogeneous) strong martingale problem introduced in Definition \ref{DSMPMod}.

\begin{theorem}
\label{thMarkov}
We denote
 $$ \Da = \lbrace g:[0,T] \to D(L) \; \text{such that assumptions 
\ref{lemMA1} and \ref{lemMA2} of Lemma \ref{lemmaMarkov} 
are fulfilled} \rbrace  $$ 
and, for $g \in \Da$, $\;\a(g)(t,s) = \frac{\partial g}{\partial t}(t,s) + 
Lg(t,\cdot)(s), \; \forall t \in [0,T], s \in \R.$

Then $S$ is a solution of the strong martingale problem introduced in Definition \ref{DSMPMod}.
\end{theorem}

\begin{remark}
\label{DALMarkov}
Let $g \in \Da$. Since for each $t \in [0,T]$, by assumptions \ref{lemMA1}
 and \ref{lemMA2} of Lemma \ref{lemmaMarkov}, $\a(g)(t,\cdot) \in E$, then, obviously $\a(g) \in \cL$.
Moreover, the same assumptions imply that 
$t\mapsto \frac{\partial g}{\partial t} (t,\cdot)$ and $t\mapsto Lg(t,\cdot)$ are continuous on $[0,T]$ and hence are bounded, i.e.
$
\sup_{t\in[0,T]} \norm{\frac{\partial g}{\partial t} (t,\cdot)}_E < \infty,\quad 
\sup_{t\in[0,T]} \norm{Lg(t,\cdot)}_E < \infty.
$ This yields in particular that Condition \eqref{aFiniteMarkov} is fulfilled.
\end{remark}

\begin{proof}[Proof of Theorem \ref{thMarkov}]\

It remains to show the martingale property \eqref{SMgDecompMarkov}.
We fix $0\le u<t\leq T$ and  a bounded random variable $\cF_u$-measurable
$G$. It will be sufficient to show that
\begin{equation}
\label{proofMg}
\E{A(u,t)} = 0,
\end{equation} where $
A(u,t) = G \left( g(t,S_t) - g(u,S_u) - \int_{u}^{t} \partial_r g(r,S_r)dr -
  \int_{u}^{t} L g(r,.)(S_r)dr \right)$.
By taking the conditional expectation of $A(u,t)$ with respect to $\cF_u$, 
using \eqref{MarkovDef} and  Fubini's theorem, we get
$
\E{A(u,t) | \cF_u} = G \phi(S_u),
$ where
$
\phi(x) = \left( P_{t-u}g(t,.) (x) - g(u,x) -  \int_{u}^{t} 
(P_{r-u}\partial_r g(r,.))(x)dr - \int_{u}^{t} (P_{r-u}L g(r,.)) (x)dr \right),
 \; \forall x\in\R.
$
We define $f:[0,T-u] \times \R \rightarrow \R$ 
by $f(\tau, \cdot) = g(\tau+u, \cdot)$.
$f$ fulfills the assumptions of Lemma \ref{lemmaMarkov}
with $T$ being replaced by $T-u$.
 By the change of variable $v=r-u$,
setting $\tau = t-u$,
the equality above becomes
$
\phi(x) = \left( P_{\tau}f(\tau,.) (x) - f(0,x) -  \int_{0}^{\tau} 
(P_{v}\partial_r f(v,.))(x)dv - \int_{0}^{\tau} (P_{v}L f(v,.)) (x)dr \right).
$
Now by Lemma \ref{lemmaMarkov} 
 we get that $\phi(x)=0, \; \forall x\in\R$. Consequently $ \E{A(u,t) | \cF_u} = 0$ and \eqref{proofMg} is fulfilled.
\end{proof}

\begin{remark} \label{R23Markov}
We introduce the following subspace $E^2_0$ of $C^2$.
\begin{equation}
\label{SetE2_0}
E^2_0=\lbrace f \in C^2 \; \text{such that $f''$ vanishes at infinity} \rbrace. 
\end{equation}
Notice that only the second derivative is supposed to vanish at infinity.
  $E^2_0$ is included in $  E$.
Indeed, if $f\in E^2_0$,  then the Taylor expansion $f(x) = f(0) + x f^\prime(0) + x^2 \int_0^1 (1-\alpha) f''(x\alpha) d\alpha$ implies that $\widetilde{f}$ is bounded.
On the other hand, by straightforward calculus we see that the first derivative $\frac{d \widetilde{f}}{dx}$ is bounded. This implies that $\widetilde{f}$ is uniformly continuous. 
 In several examples it is easy to identify $E^2_0$ as a significant subspace of $D(L)$, see for instance the 
example of L\'evy processes which is described below.
\end{remark}

\subsubsection{A significant particular case: L\'evy processes}
\label{SSLevy}

As anticipated above, an insightful example for Markov flows is the case of L\'evy processes. We specify below  the corresponding infinitesimal generator. 

Let $(\Lambda_t)$ be a square integrable L\'evy process with characteristic triplet $(A, \nu, \gamma)$, such that $\Lambda_0=0$. We refer to, e.g., \cite[Chapter 3]{ContTankovBook} for more details.

We suppose that $(\Lambda_t)$ is of pure jump, i.e. $A=0$ and $\gamma=0$. 
Since $\Lambda$ is square integrable, then (cf. \cite[Proposition 3.13]{ContTankovBook})
\begin{equation}
\label{LevySqInt}
\int_{\R} |s|^2 \nu(ds) < \infty
\end{equation}
and
\begin{equation}
\label{LevySqInt12}
c_1 := \frac{\E{\Lambda_t}}{t} = \int_{|s|>1} s \nu(ds) < \infty, \quad c_2 := 
\frac{\mathrm{Var}[\Lambda_t]}{t} = \int_{\R} |s|^2 \nu(ds) < \infty.
\end{equation}
Clearly the corresponding Markov flow is given by $X_t^{0,x} = x + \Lambda_t, t \ge 0, x\in \R$.

The classical theory of semigroup for L\'evy processes
is for instance developed in Section 6.31 of  \cite{SatoBook}.
There   one defines the
 semigroup $P$ on the set $C_0$ of continuous functions vanishing at
 infinity, equipped with the sup-norm $\norm{u}_\infty = \sup_s{|u(s)|}$.
 By  \cite[Theorem 31.5]{SatoBook}, 
the semigroup  $P$ is strongly continuous  on $C_0$, 
 with norm $\norm{P}=1$, and its generator $L_0$ is given by 
\begin{equation}
\label{L0}
L_0f(s) = \int \left(f(s+y)-f(s)-yf^\prime(s)\1_{|y|<1}\right)\nu(dy).
\end{equation}
Moreover, 
the set $C^2_0$  of functions $f\in C^2$ such that $f$, $f^\prime$ and
 $f^{''}$ vanish at infinity,
is included in $D(L_0)$.
We remark that the domain $D(L)$ includes the classical domain $D(L_0)$.
In fact, we have
$ \Vert g \Vert_E \le \Vert g \Vert_{C_0}, \ \forall g \in C_0.$ 
Consequently,  if $f\in D(L_0) \subset C_0$, then for $t>0$
$$\norm{\dfrac{P_t f-f}{t}  - L_0 f}_E \leq  \norm{\dfrac{P_t f-f}{t} -
L_0 f }_{C_0}.$$
So $f\in D(L)$ and $Lf = L_0 f$.
Assumption \ref{A_Markov} is verified because of Proposition \ref{PropMarkov}, item \ref{R25i} of Remark \ref{R25} 
and Remark \ref{RStrongCont}.

The theorem below shows that the space $E^2_0$,
defined in Remark \ref{R23Markov}, is a subset of $D(L)$.
\begin{theorem}
\label{thLevyGen}
Let $L$ be the infinitesimal generator of the semigroup $(P_t)$.
Then $E^2_0\subset D(L)$ and 
\begin{equation}
\label{L_Levy}
Lf(s) = \int \left(f(s+y)-f(s)-yf^\prime(s)\1_{|y|<1}\right)\nu(dy), \ f \in E^2_0.
\end{equation} 
\end{theorem}

A proof of this result, using arguments in \cite{Figueroa2008small}, is 
developed in Appendix \ref{appA_BSDE}.

In conclusion, $C^2$ functions whose  second derivative vanishes at infinity belong to $D(L)$. 
For instance, $id:s \mapsto s \in D(L)$.  On the other hand the function 
$s \mapsto s^2$
also belongs to $D(L)$.

In fact, for every $s\in \R, t\geq 0$ we have
$$
P_t f(s) - f(s) = \dfrac{\E{(s+\Lambda_t)^2} - s^2}{t} = \dfrac{2s c_1t + c_2t +c_1^2 t^2}{t}.
$$
Obviously, this converges to the function $s \mapsto  2s c_1 + c_2$ according to the $E$-norm.
Finally it follows that $L(s^2) = 2s c_1 + c_2$.

\begin{corollary} \label{CGenerator}
We have the  inclusion
$$
E^2_0 \oplus \{s \mapsto s^2 \} \subset  D(L)
$$
\end{corollary}

\subsection{Diffusion processes}
\label{E22}

Here we will suppose again $\cO = \R\times E$, where $E=\R$ or $]0,\infty[$.
A function $f: [0,T] \times \cO$ will be said to be {\bf globally Lipschitz} if
it is  Lipschitz with respect to the second and third variable
uniformly with respect to the first. 

We consider here the case of a diffusion process $(X,S)$ whose dynamics is described as follows:
\begin{align}
\label{expleDiff}
\begin{split}
d X_t &= b_X(t,X_t,S_t) dt + \Sum_{i=1}^{2}\sigma_{X,i}(t,X_t,S_t) d W^i_t\\
d S_t &= b_S(t,X_t,S_t) dt + \Sum_{i=1}^{d}\sigma_{S,i}(t,X_t,S_t) d W^i_t, 
\end{split}
\end{align}
 where $W = (W^1,W^2)$ is a standard two-dimensional Brownian motion
with canonical filtration $(\cF_t)$, 
$b_X$, $b_S$, $\sigma_{X,i}$, and $\sigma_{S,i}$, for $i=1,2$, 
$b, \sigma: [0,T] \times \R^2 \rightarrow \R$
are continuous functions which  are globally Lipschitz.

We also suppose $(X_0,S_0)$  to have all moments and that $S$ takes value in $E$. For instance a geometric Brownian motion takes value in $E=]0,\infty[$,
if it starts from a positive point.

\begin{remark}\label{Rallmoments}
Let $ p \ge 1$.
It is well-known, that there is 
a constant $C(p)$, only depending on $p$, such that
$$ \E{\sup_{t \le T} (\vert X_t \vert^ p + \vert S_t \vert^ p)} \le
C(p) ( \vert X_0 \vert ^ p + \vert S_0 \vert ^ p).$$
\end{remark}

By It\^o formula, for 
$f \in \mathcal{C}^{1,2}([0,T[\times \cO) 
 $, we have
\begin{eqnarray*}
d f(t,X_t,S_t) &=& \partial_t f(t,X_t,S_t) dt + \partial_s f(t,X_t,S_t) dS_t + \partial_x f(t,X_t,S_t) dX_t \\
&+&  \frac{1}{2} \left\lbrace \partial_{ss} f(t,X_t,S_t) d[S]_t +
 \partial_{xx} f(t,S_t,X_t) d[X]_t + 2 \partial_{sx} f(t,X_t,S_t) d[S, X]_t \right\rbrace. 
\end{eqnarray*}
We denote
$|\sigma_S|^2 = \Sum_{i=1}^{2}\sigma_{S,i}^2$, $\;|\sigma_X|^2 = 
\Sum_{i=1}^{2}\sigma_{X,i}^2$ and $\langle \sigma_S, \sigma_X\rangle = \Sum_{i=1}^{2}\sigma_{S,i}\sigma_{X,i}$.

Hence, the operator $\a$ can be defined as
$$
\a(f) = \partial_t f + b_S \partial_s f + b_X \partial_x f +  \frac{1}{2} \left\lbrace |\sigma_S|^2\partial_{ss} f + |\sigma_X|^2 \partial_{xx} f + 2 \langle \sigma_S, \sigma_X\rangle \partial_{sx} f  \right\rbrace, 
$$ associated with 
$A_t \equiv t$ and a domain 
$\Da =  \mathcal{C}^{1,2}([0,T[\times \cO) \cap 
\mathcal{C}^{1}([0,T]\times \cO)$. \\
Notice  that Assumption \ref{E0} is verified since 
$id$ and $id \times id$ belong to $\Da$.
Moreover, a straightforward calculation gives 
$$
\widetilde{\a}(f) = |\sigma_S|^2 \partial_s f(t,x,s) + \langle \sigma_S, \sigma_X\rangle \partial_x f(t,x,s)
$$
In particular, $\;\widetilde{\a}(id) = |\sigma_S|^2.$

\begin{remark} \label{RE22} By It\^o formula, for $0\leq u \leq T$, we obviously have
\begin{align*}
\begin{split}
f(u,X_u,S_u) - \int_0^u \a(f)(r,X_r,S_r) dr &=
\int_0^u \partial_x f(r,X_r,S_r) \left(\sigma_{X,1}(r,X_r,S_r) d W^1_r+\sigma_{X,2}(r,X_r,S_r) d W^2_r\right) \\
&+ \int_0^u \partial_s f(r,X_r,S_r) \left(\sigma_{S,1}(r,X_r,S_r) d W^1_r+\sigma_{S,2}(r,X_r,S_r) d W^2_r\right).
\end{split}
\end{align*}
\end{remark}

\subsection{Variant of diffusion processes}

Let $(W_t)$ be an $\cF_t$-standard Brownian motion
and $S$ be a solution of the SDE 

\begin{equation}
\label{exp3}
d S_t = \sigma(t,S_t) dW_t + b_1(t,S_t)da_t + b_2(t,S_t)dt,
\end{equation}
where $b_1, b_2, \sigma: [0,T] \times \R^2 \rightarrow \R$
are continuous functions which  are globally Lipschitz, and
 $a: [0,T] \rightarrow \R$ is an increasing function  
such that $da$ is singular with respect to Lebesgue measure.
 We set $A_t = a_t + t$.

The equation \eqref{exp3} can be written as
$$
d S_t = \sigma(t,S_t) dW_t + \left( b_1(t,S_t)\frac{d\a_t}{dA_t} +
 b_2(t,S_t)\frac{dt}{dA_t} \right) dA_t. 
$$

A solution $S$ of \eqref{exp3} verifies the strong martingale problem related to $(\Da, \a, A)$ 
with $A_t = t$, where $\Da=\mathcal{C}^{1,2}([0,T]\times \R)$ and for $f\in \Da$,
$$
 \a(f)(t,s) =  \left( \partial_t f(t,s) \frac{dt}{dA_t} + \partial_s f(t,s)
 \widetilde{b}(t,s) + \frac{1}{2}\partial_{ss} f(t,s) 
\widetilde{\sigma}^2(t,s) \right),
$$
where 
 $\widetilde{b}(t,s) = b_1(t,s) \frac{d\a_t}{dA_t}(t) + b_2(t,s)
\frac{dt}{dA_t}(t)$ and 
$\widetilde{\sigma}^2(t,s) = \sigma^2(t,s)\frac{dt}{dA_t}(t)$.

Indeed, by It\^o formula, the process $\; t \mapsto f(t,S_t) - \int_0^t  \a (f)(r,S_r) dA_r $
 is a local martingale.

\subsection{Exponential of additive processes}
\label{expleExpAdd}

A c\`adl\`ag process $(Z^1,Z^2)$ is said to be an {\bf additive process} if  $(Z^1,Z^2)_0=0$, $(Z^1,Z^2)$ is continuous
 in probability and it has independent increments, i.e. $(Z^1_t-Z^1_u,Z^2_t-Z^2_u)$ is independent of $\cF_u$ for $0 \leq u \leq t \leq T$ 
and $(\cF_t)$ is the canonical filtration associated with $(Z^1,Z^2)$.

In this section we restrict ourselves to the case of exponential of  additive processes which are semimartingales
(shortly semimartingale additive processes)  and we specify a corresponding 
martingale problem $(\a, \Da, A)$ for this process. This will be based on Fourier-Laplace transform techniques.
The couple of processes $(X, S)$ is defined by
$$ X = \exp(Z^1) \quad
S =  \exp(Z^2),
$$
where $(Z^1, Z^2)$ is a  semimartingale additive process taking values in $\R^2$. \\
We denote by $D$ the set
$$
D :=\lbrace z=(z_1, z_2)\in\C^2 |\quad \E{ |X_T^{{\rm Re}(z_1)} S_T^{{\rm Re}(z_2)}|}<\infty \rbrace.
$$
We convene that $\C^2 = \R^2 + i  \R^2$,
associating the couple $(z_1,z_2)$ with
$({\rm Re} z_1,  {\rm Re} z_2) + i ({\rm Im} z_1,  {\rm Im} z_2)$.
Clearly we have $D = (D \cap \R^2) + i \R^2$.
We also introduce  the notation
$$
D/2 := \lbrace z \in\C^2| \quad 2z\in D\} \subset D .
$$ 
\begin{remark} \label{R23}
By Cauchy-Schwarz inequality, $z,y \in D/2$ implies that $z+y\in D$.
\end{remark}


We denote by  $\kappa: D \rightarrow \C $,  the
 generating function of $(Z^1, Z^2)$, see for instance \cite[Definition 2.1]{gor2013variance}.
In particular
 $\kappa$ verifies
$
\exp(\kappa_t(z_1, z_2)) = \E{ \exp(z_1 Z_t^1 + z_2 Z_t^2)} = \E{ X_t^{z_1} S_t^{z_2}}.
$
We will adopt  similar notations and assumptions  as in \cite{gor2013variance},
 which treated the problem of variance optimal hedging for
a one-dimensional  exponential of additive process.
We  introduce 
 a  function $\rho$,  defined, for each $t\in[0,T]$, as follows:
\begin{eqnarray}
\label{ExpAddRhoS}
\rho_t(z_1,z_2,y_1,y_2) &:=& \kappa_t(z_1+y_1, z_2+y_2) - \kappa_t(z_1, z_2) - \kappa_t(y_1, y_2), \quad \text{ for } (z_1,z_2),(y_1,y_2)\in D/2,  \nonumber\\
\rho_t(z_1,z_2) &:=& \rho_t(z_1,z_2,\bar{z_1},\bar{z_2}), \quad \text{ for }  (z_1,z_2)\in D/2,  \\
\rho^S_t &:=& \rho_t(0,1) = \kappa_t(0,2) - 2\kappa_t(0,1), \text{ if } (0,1) \in D/2.  \nonumber
\end{eqnarray}
We remark that for $(z_1,z_2)\in D/2$, $t\mapsto \rho_t(z_1,z_2)$ is a real function.
These functions appear naturally in the expression of the angle brackets of  $(M^X,M^S)$
where $M^X$ (resp. $M^S$) is the martingale part of $X$ (resp. $S$).

From now on, in this section, the assumption below will be in force.
\begin{assumption}\
\label{A_PAI}
\begin{enumerate}[label=\arabic*)]
\item $\rho^S$ is strictly increasing.
\item $(0,2) \in D$.
\end{enumerate}
\end{assumption}
Notice that  item 2) is equivalent to the existence of the second order moment of $S_T$. 
Moreover, 2) implies,  by Cauchy-Schwarz, that $D/2 + (0,1) \subset D$.

We recall that previous assumption implies that $Z^2$ has no deterministic increments,
see \cite[Lemma 3.9]{gor2013variance}.

Similarly as in \cite[Propositions 3.4 and 3.15]{gor2013variance}, one can prove the following result.
\begin{proposition}
\label{propAdd}\
\begin{enumerate}[label=\arabic*)]
\item \label{propAddPt1} For every $(z_1,z_2)\in D$, $\left(X_t^{z_1}S_t^{z_2}e^{-\kappa_t(z_1,z_2)}\right)$ is a martingale. 
\item \label{propAddPt2} $t\mapsto \kappa_t(z_1,z_2)$ is a bounded variation continuous function, for every $(z_1,z_2)\in D$.
In particular, $t \mapsto \rho_t(z_1,z_2)$ is also a  bounded variation continuous function, 
for every $(z_1,z_2)\in D/2$.
\item \label{propAddPt3} Let $I$ be a compact real set included in $D$. Then 
$$\sup_{(x,y)\in I}\sup_{t\leq T} \E{X_t^x S_t^y}= \sup_{(x,y)\in I}\sup_{t\leq T} e^{\kappa_t(x,y)}< \infty.$$
\item \label{propAddPt4} $\forall (z_1,z_2)\in D/2$, $t\mapsto \rho_{t}(z_1,z_2)$ is non-decreasing. 
\item \label{propAddPt5} $ \kappa_{dt}(z_1,z_2) \ll \rho^S_{dt}  \text{ , for every } z\in D$.
\item \label{propAddPt6} $ \rho_{dt}(z_1,z_2,y_1,y_2) \ll \rho^S_{dt}  \text{ , for every } (z_1,z_2),(y_1,y_2)\in D/2$.
\end{enumerate}
\end{proposition}

\begin{remark} \label{R13}
Notice that, for any $(z_1,z_2) \in D$, $X^{z_1} S^{z_2}$ is a special semimartingale. Indeed, by Proposition \ref{propAdd}, $X_t^{z_1} S_t^{z_2} = N_t e^{\kappa_t(z_1,z_2)}$ where 
$\kappa(z_1,z_2)$ is a bounded variation continuous function and $N$ is a martingale. Hence, integration by parts implies that $X^{z_1} S^{z_2}$ is a special semimartingale whose decomposition is given by
\begin{equation} \label{ER13}
X^{z_1} S^{z_2}= M(z_1,z_2) + V(z_1,z_2),
\end{equation}
 where $M_t(z_1,z_2)= X_0^{z_1} S_0^{z_2} + \int_0^t e^{\kappa_u(z_1,z_2)} dN_u$ and $V_t(z_1,z_2)= \int_0^t X^{z_1}_{u-} S^{z_2}_{u-} \kappa_{du}(z_1,z_2)$, $\forall t\in[0,T]$.
\end{remark}

The  proposition below shows that the local martingale part of the decomposition above is
a square integrable martingale if $(z_1,z_2) \in D/2$ and gives its angle bracket in terms of the generating function.
\begin{proposition}
\label{ExpAddSqMg}
Let $z=(z_1,z_2), y=(y_1,y_2) \in D/2$.  Then $X^{z_1} S^{z_2}$ is a special semimartingale, whose decomposition $X^{z_1} S^{z_2}= M(z_1,z_2) + V(z_1,z_2)$ satisfies, for $t\in[0,T]$,
\begin{eqnarray*}
V(z_1,z_2)_t &=& \int_0^t X^{z_1}_{u-} S^{z_2}_{u-} \kappa_{du}(z_1, z_2) \\
\langle M(z_1,z_2), M(y_1,y_2) \rangle_t &=& \int_0^t X^{z_1+y_1}_{u-} S^{z_2+y_2}_{u-} \rho_{du}(z_1, z_2,y_1, y_2).
\end{eqnarray*} In particular, 
\begin{equation*}
\langle M(z_1,z_2)\rangle_t := \langle M(z_1,z_2), \overline{M(z_1,z_2)} \rangle_t = \int_0^t X^{2{\rm Re}(z_1)}_{u-} S^{2{\rm Re}(z_2)}_{u-} \rho_{du}(z_1,z_2).
\end{equation*}
Moreover, $M(z_1,z_2)$ is a square integrable martingale.
\end{proposition}
\begin{proof}
This can be done adapting the techniques of \cite[Lemma 3.2]{Hubalek2006} 
and its generalization to one-dimensional additive processes, i.e. \cite[Proposition 3.17 and Lemma 13.19]{gor2013variance}.
\end{proof}

The measure $d \rho^S$, called \textbf{reference variance measure} in \cite{gor2013variance}, plays a central role in the expression of the canonical
 decomposition of special semimartingales depending on the couple $(X,S)$.
\begin{corollary}
\label{CCS}
The semimartingale decomposition of $S$ is given by
 $S=M^S+V^S $, where, for $t\in[0,T]$
$$ V^S_t = \int_0^t S_{u-} \kappa_{du}(0,1) \quad
\langle M^S \rangle_t = \int_0^t S_{u-}^2 \rho^S_{du}.
$$
\end{corollary}
\begin{proof}
It follows from Proposition \ref{ExpAddSqMg} setting $z_1 = 0, z_2=1$.
\end{proof}

Now we state some  useful estimates.
\begin{lemma}
\label{expAddMoments}
Let $(a,b)\in D \cap \R^2$. Then $\E{\sup_{t\leq T} X_t^a S_t^b} < \infty.$
\end{lemma}

\begin{proof}
Let $(a,b)\in D \cap \R^2$, then $(a/2,b/2)\in D/2$. 
By Proposition \ref{ExpAddSqMg}, we have $$X_t^{a/2} S_t^{b/2}= M_t(a/2,b/2) + \int_0^t X^{a/2}_{u-} S^{b/2}_{u-} \kappa_{du}(a/2,b/2), \; \forall t \in [0,T]$$ 
and  $M(a/2,b/2)$ is a square integrable martingale. Hence, by Doob inequality, we have
$$
\E{\sup_{t\leq T} \left| M_t(a/2,b/2)\right|^2 } \leq 4 \E{\left| M_T(a/2,b/2)\right|^2 }<\infty.
$$
On the other hand, using Cauchy-Schwarz inequality and Fubini theorem, we obtain
\begin{eqnarray*}
\E{\sup_{t\leq T} \left| \int_0^t X^{a/2}_{u-} S^{b/2}_{u-} \kappa_{du}(a/2, b/2)\right|^2} &\leq &  \norm{\kappa(a/2, b/2)}_{T} \int_0^T \E{X^{a}_{u-} S^{b}_{u-}} \norm{\kappa(a/2, b/2)}_{du} \\
&= & \norm{\kappa(a/2, b/2)}_{T} \int_0^T e^{\kappa_u(a, b)} \norm{\kappa(a/2, b/2)}_{du}\\
&\leq & e^{\norm{\kappa(a, b)}_T} \norm{\kappa(a/2, b/2)}_{T}^2 < \infty.
\end{eqnarray*}
Finally  
$ \E{\sup_{t\leq T} X_t^{a} S_t^{b} } = \E{\sup_{t\leq T} \left| X_t^{a/2} S_t^{b/2} \right|^2}< \infty.$
\end{proof}

In the general case, when $(z_1,z_2) \in D$, the local martingale part of the special semimartingale $X^{z_1} S^{z_2}$ is a true  (not necessarily square integrable) martingale.
\begin{proposition}
\label{ExpAddMg}
Let $(z_1,z_2) \in D$, then, $M(z_1,z_2)$, the local martingale part of $X^{z_1} S^{z_2}$, is a true martingale such that $
\E{\sup_{t\leq T} \left| M_t(z_1,z_2)\right| } < \infty .$
\end{proposition}

\begin{proof}
Let $(z_1,z_2) \in D$. Adopting the notations of \eqref{ER13}, we recall that,
by Proposition \ref{ExpAddSqMg}, $\forall t\in [0,T]$, 
$M_t(z_1,z_2)= X_t^{z_1} S_t^{z_2} - \int_0^t X^{z_1}_{u-} S^{z_2}_{u-} \kappa_{du}(z_1,z_2)$.
For this local martingale we can write 
\begin{eqnarray*}
\E{\sup_{t\leq T} \left| M_t(z_1,z_2)\right| } &\leq & \E{\sup_{t\leq T} \left| X_t^{z_1} S_t^{z_2}\right|} + \E{\int_0^T \left|X^{z_1}_{t-} S^{z_2}_{t-}\right| \norm{\kappa(z_1,z_2)}_{dt}} \\
&\leq & \E{\sup_{t\leq T} \left| X_t^{{\rm Re}(z_1)} S_t^{{\rm Re}(z_2)}\right|} 
\left( 1 + \norm{\kappa(z_1,z_2)}_{T} \right).
\end{eqnarray*} 
Since $({\rm Re}(z_1),{\rm Re} (z_2))$ belongs to $D$, by
 Lemma \ref{expAddMoments}, the right-hand side is finite. 
Consequently the local martingale $M(z_1,z_2)$ is indeed a true martingale.
\end{proof}

The goal of this section is to show that $(X,S)$ is a solution of a strong martingale problem,
with related  triplet  $(\Da, \a, A)$, which will be specified below. 
For this purpose, we determine the semimartingale decomposition 
of $(f(t,X_t,S_t))$ for 
functions $f:[0,T]\times \cO \rightarrow \C$, where $\cO=]0,\infty[^2$, of the form
\begin{equation} \label{DFLaplace}
f(t,x,s) := \int_{\C^2} d \Pi(z_1, z_2) x^{z_1} s^{z_2} \lambda(t,z_1,z_2),\; \forall t\in[0,T], x,s> 0,
\end{equation}
where $\Pi$ is a finite complex Borel measure on $\C^2$ and $\lambda: [0,T]\times \C^2\longrightarrow \C$. 
The family of those functions will include the set $\Da$ defined later. 

 Proposition \ref{ExpAddMg} and item \ref{propAddPt5} of Proposition \ref{propAdd} say that, for $z=(z_1,z_2) \in D$,
$$
t \mapsto X_t^{z_1} S_t^{z_2} - \int_0^t X^{z_1}_{u-} S^{z_2}_{u-} \kappa_{du}(z_1, z_2)
= X_t^{z_1} S_t^{z_2} - \int_0^t X^{z_1}_{u-} S^{z_2}_{u-} \frac{d\kappa_{u}(z_1, z_2)}{d \rho^S_u} \rho^S_{du}$$ is a martingale. This provides the semimartingale decomposition of
 the basic functions 
$(t,x,s) \mapsto x^{z_1} s^{z_2}$ for $z_1,z_2 \in D$, applied to $(X,S)$.
Those functions are expected to be elements of
 $\Da$ and  one candidate for the bounded variation process $A$
 is $\rho^S$.
It remains to precisely define $\Da$ and the operator $\a$.

A first step in this direction is to  consider a Borel
 function $\lambda: [0,T]\times \C^2 \rightarrow \C$ such that, for any $(z_1,z_2)\in D$, $t\in[0,T] \mapsto \lambda(t, z_1, z_2)$ is absolutely continuous with respect to $\rho^S$.

\begin{lemma}
\label{ExpAddLemma}
Let $\lambda: [0,T]\times \C^2 \rightarrow \C$ such that, for any $(z_1,z_2)\in D$, $t\in[0,T] \mapsto \lambda(t, z_1, z_2)$ is absolutely continuous with respect to $\rho^S$.
Then for any $(z_1,z_2)\in D$, 
\begin{equation}
\label{MLambda}
t \mapsto M^\lambda_t(z_1,z_2) := S_t^{z_1} X_t^{z_2}\lambda(t,z_1,z_2) - \int_0^t X_{u-}^{z_1} S_{u-}^{z_2} \left\lbrace \dfrac{d\lambda(u,z_1,z_2)}{d\rho^S_u} + \lambda(u,z_1,z_2)\dfrac{d\kappa_u(z_1,z_2)}{d\rho^S_u} \right\rbrace \rho^S_{du},
\end{equation}
is a martingale. Moreover, if $(z_1,z_2)\in D/2$ then $M^\lambda(z_1,z_2)$ is a square integrable martingale and
\begin{equation}
\label{MLambda2Moment}
\E{|M^\lambda_t(z_1,z_2)|^2} = \int_0^t e^{\kappa_u(2{\rm Re}(z_1), 2{\rm Re}(z_2))} |\lambda(u,z_1,z_2)|^2 \rho_{du}(z_1,z_2).
\end{equation}
\end{lemma}

\begin{proof}
Let $(z_1,z_2) \in D$, $M(z_1,z_2)$ and $V(z_1,z_2)$ be the random fields introduced in Remark \ref{R13}. Since $\lambda(dt, z_1, z_2) \ll \rho^S_{dt}$, then $t\mapsto \lambda(t,z_1,z_2)$ is a bounded continuous function on $[0,T]$.
By item \ref{propAddPt5} of Proposition \ref{propAdd} $M^\lambda(z_1,z_2)$ is well-defined. 
Integrating by parts and taking into account Remark \ref{R13} 
allows to show 
\begin{equation} \label{EMLambda} 
M^\lambda_t(z_1,z_2) = \lambda(0,z_1,z_2) M_0(z_1,z_2) + \int_0^t \lambda(u,z_1,z_2) dM_u(z_1,z_2),\; \forall t\in[0,T].
\end{equation}
 Obviously   $M^\lambda(z_1,z_2)$ is  a local martingale. 
In order to prove that it is a true martingale, we establish that 
$$
\E{\sup_{t\leq T} \left| M^\lambda_t(z_1,z_2) \right| } < \infty.
$$
Indeed, by integration by parts in \eqref{EMLambda},  for $t\in [0,T]$
we have
$$
M^\lambda_t(z_1,z_2) = \lambda(t,z_1,z_2) M_t(z_1,z_2) -\int_0^t  M_{u-}(z_1,z_2)\lambda(du,z_1,z_2).
$$
Hence, as in the proof of Lemma \ref{expAddMoments},
\begin{align}
\label{EMartlambda}
\begin{split}
\E{\sup_{t\leq T} \left| M^\lambda_t(z_1,z_2) \right| } \leq &
\E{\sup_{t\leq T} \left|\lambda(t,z_1,z_2) M_t(z_1,z_2)\right|} + \E{\int_0^T \left|M_{u-}(z_1,z_2)\right| \norm{\lambda(.,z_1,z_2)}_{dt}} \\
 \leq & 2 \E{\sup_{t\leq T} \left| M_t(z_1,z_2)\right|} \norm{\lambda(.,z_1,z_2)}_T.
\end{split}
\end{align}
 Thanks to Proposition \ref{ExpAddMg}, the right-hand side of \eqref{EMartlambda} is finite
 and finally 
$ M^\lambda(z_1,z_2)$ is shown to be a martingale so that the first part of Lemma \ref{ExpAddLemma}
is proved. 

Now, suppose that $(z_1,z_2) \in D/2$. By \eqref{EMLambda} and 
 Proposition \ref{ExpAddSqMg}, we have
\begin{align}
\label{E224}
\begin{split}
\E{ \langle M^\lambda(z_1, z_2) \rangle_T } &= \E{\int_0^T |\lambda(u,z_1,z_2)|^2\langle M(z_1, z_2)\rangle_{du}}  \\
&= \E{\int_0^T X_{u-}^{2 {\rm Re}(z_1)} S_{u-}^{2{\rm Re}(z_2)}|\lambda(u,z_1,z_2)|^2 \rho_{du}(z_1,z_2)} \\
&= \int_0^T e^{\kappa_u(2{\rm Re}(z_1), 2{\rm Re}(z_2))} |\lambda(u,z_1,z_2)|^2 \rho_{du}(z_1,z_2)\\
\leq & \sup_{u\leq T} e^{\kappa_u(2{\rm Re}(z_1), 2{\rm Re}(z_2))} \int_0^T |\lambda(u,z_1,z_2)|^2 \rho_{du}(z_1,z_2) < \infty. 
\end{split}
\end{align}
The latter term is finite by  point \ref{propAddPt3} of 
Proposition \ref{propAdd} and by the fact that
 $\lambda(.,z_1,z_2)$ is bounded on $[0,T]$.
Consequently, $M^\lambda(z_1,z_2)$ is a square integrable martingale and since $|M^\lambda(z_1,z_2)|^2 - \langle M^\lambda(z_1, z_2) \rangle$ is a martingale, 
the estimate \eqref{E224} yields the desired identity \eqref{MLambda2Moment}.
\end{proof}

Now, let $\Pi$ be a finite Borel measure on $\C^2$ and let us formulate
 the following assumption on it.
\begin{assumption}
\label{A_Pi}
We set $I_0:={\rm Re}(supp\; \Pi)$. 
\begin{enumerate}
\item $I_0$ is bounded.
\item $I_0 \subset D.$
\end{enumerate}
\end{assumption}
Notice that this assumption implies that $supp\;\Pi \subset D$.

\begin{theorem} \label{T17}
Suppose that Assumptions \ref{A_PAI} and \ref{A_Pi}  are verified.
Let $\lambda: [0,T]\times \C^2 \rightarrow \C$ be a function such that
\begin{eqnarray}
\forall(z_1,z_2)\in supp\;\Pi, \; \lambda(dt, z_1, z_2) &\ll& \rho^S_{dt}, \label{CondLambda0}\\
\forall t\in[0,T], \; \int_{\C^2}d|\Pi|(z_1,z_2) |\lambda(t,z_1,z_2)|^2 &<& \infty,\label{CondLambda1}\\
\int_0^T d\rho^S_t \int_{\C^2}d|\Pi|(z_1,z_2)  \left|  \dfrac{d\lambda(t,z_1,z_2)}{d\rho^S_t} + \lambda(t,z_1,z_2) \dfrac{d\kappa_t(z_1,z_2)}{d\rho^S_t}\right| &<& \infty. \label{CondLambda3}
\end{eqnarray}
Then the function $f$ of the form \eqref{DFLaplace}
is continuous. Moreover 
\begin{equation}
\label{EMartFubini}
t \mapsto \widehat{M}^\lambda_t :=f(t,X_t,S_t) - \int_0^t \rho^S_{du} \int_{\C^2} d \Pi(z_1, z_2) X_{u-}^{z_1} S_{u-}^{z_2} \left\lbrace  \dfrac{d\lambda(u,z_1,z_2)}{d\rho^S_u} + \lambda(u,z_1,z_2) \dfrac{d\kappa_u(z_1,z_2)}{d\rho^S_u} \right\rbrace,
\end{equation}
 is a martingale.
\end{theorem}
\begin{remark} \label{R226} \
In  \eqref{CondLambda3} and  \eqref{EMartFubini}, part of the integrand with respect to $\Pi$
is only defined for $(z_1,z_2) \in D$. By convention the integrand 
will be set to zero for $(z_1,z_2) \notin D$. 
In the sequel we will adopt the same rule.
\end{remark}
\begin{proof}
Let $\lambda: [0,T]\times \C^2 \rightarrow \C$ verifying the hypotheses of the theorem.
The function $f$ is well-defined. Indeed, for $t\in[0,T], x,s> 0$,
$$|f(t,x,s)| \leq  \sup_{(a,b)\in I_0}x^{a} s^{b}  \int_{\C^2} d |\Pi|(z_1, z_2) |\lambda(t,z_1,z_2)|, $$
which is finite because of  Condition \eqref{CondLambda1} and Assumption \ref{A_Pi}, 
taking into account Cauchy-Schwarz inequality.
Moreover, by Fubini theorem and using the definition of $f$ in \eqref{DFLaplace}, we get
\begin{eqnarray}
\label{EfFinite}
\E{|f(t,X_t,S_t)|} & \leq & \int_{\C^2} d |\Pi|(z_1, z_2) \E{X_t^{{\rm Re}(z_1)}S_t^{{\rm Re}(z_2)}} |\lambda(t,z_1,z_2)| \nonumber\\
&\leq & \sup_{u\in[0,T], (a,b)\in I_0} \E{ X_{u}^{a} S_{u}^{b}} \int_{\C^2} d |\Pi|(z_1, z_2)|\lambda(t,z_1,z_2)|.
\end{eqnarray} The right-hand side is finite by item \ref{propAddPt3} of Proposition \ref{propAdd} and Condition \eqref{CondLambda1}.

We observe that $t \mapsto \lambda(t,z_1,z_2)$ is a continuous function
since it is absolutely continuous with respect to $\rho^S$ for fixed $(z_1, z_2) \in \C^2$.
On the other hand, Condition \eqref{CondLambda1} implies that the 
family $(\lambda(t,z_1,z_2), t \in [0,T]) $ is $\vert \Pi \vert$
-uniformly integrable. These properties, together with
Lebesgue dominated convergence theorem imply that
$f$ defined in \eqref{DFLaplace} is continuous with respect to all
the variables.

We show now that  the process $t \mapsto \widehat{M}^\lambda_t$ is well-defined. This 
holds because
\begin{eqnarray}
\label{E2350}
&&\E{\int_0^t \rho^S_{du} \int_{\C^2} d |\Pi|(z_1, z_2) |X_{u-}^{z_1} S_{u-}^{z_2}| \left|  \dfrac{d\lambda(u,z_1,z_2)}{d\rho^S_u} + \lambda(u,z_1,z_2) \dfrac{d\kappa_u(z_1,z_2)}{d\rho^S_u} \right| } \\
&&\leq  \displaystyle \sup_{u\in[0,T], (a,b)\in I_0} \E{ X_{u}^{a} S_{u}^{b}} \int_0^t \rho^S_{du}\int_{\C^2} d |\Pi|(z_1, z_2)  \left|  \frac{d\lambda(u,z_1,z_2)}{d\rho^S_u} + \lambda(u,z_1,z_2) \frac{d\kappa_u(z_1,z_2)}{d\rho^S_u} \right| \nonumber,
\end{eqnarray}
which is finite by point \ref{propAddPt3} of Proposition \ref{propAdd}  and 
Condition \eqref{CondLambda3}. Inequality \eqref{E2350} allows to apply Fubini theorem to the  integral term in
\eqref{EMartFubini}, so that  
 we get
\begin{eqnarray}
\label{E153}
\widehat{M}^\lambda_t &=&  \int_{\C^2} d \Pi(z_1, z_2) \left( X_{t}^{z_1} S_{t}^{z_2} \lambda(t,z_1,z_2) - \int_0^t X_{u-}^{z_1} S_{u-}^{z_2} \left\lbrace  \dfrac{d\lambda(u,z_1,z_2)}{d\rho^S_u} + \lambda(u,z_1,z_2) \dfrac{d\kappa_u(z_1,z_2)}{d\rho^S_u} \right\rbrace \rho^S_{du} \right) \nonumber\\
&=& \int_{\C^2} d \Pi(z_1, z_2) M^\lambda_t(z_1,z_2),
\end{eqnarray}
where $M^\lambda(z_1,z_2)$ is defined in \eqref{MLambda} for any $(z_1,z_2)\in D$. We observe that 
\begin{equation} \label{E2351}
\E{ \int_{\C^2} d |\Pi|(z_1, z_2) \left|M^\lambda_t(z_1,z_2)\right|} < \infty,
\end{equation}
taking into account \eqref{EfFinite} and \eqref{E2350}.
It remains to show that $\widehat{M}^\lambda$ is a martingale. \\
Let  $0\leq s \leq t \le T$ and a bounded, $\cF_s$-measurable random variable 
$G$. By Fubini theorem and Lemma  \ref{ExpAddLemma} it follows
\begin{eqnarray*}
\E{\widehat{M}^\lambda_t G} &=& \int_{\C^2} d \Pi(z_1, z_2) \E{M^\lambda_t(z_1,z_2) G} 
= \int_{\C^2} d \Pi(z_1, z_2) \E{M^\lambda_s(z_1,z_2) G} \\
&=& \E{\int_{\C^2} d \Pi(z_1, z_2) M^\lambda_s(z_1,z_2) G} = \E{\widehat{M}^\lambda_s G},
\end{eqnarray*}
which implies the desired result.
\end{proof}

We proceed now to the definition of the domain $\Da$
in view of the specification of the corresponding martingale problem. We set
\begin{align}
\label{ExpAddADom}
\begin{split}
\Da &= \Big\{ f : (t,x,s) \mapsto \int_{\C^2} d \Pi(z_1, z_2) x^{z_1} s^{z_2} \lambda(t,z_1,z_2), \forall t\in [0,T], x,s>0, \\
&\text{where } \Pi \text{ is a finite Borel measure on } \C^2\text{ verifying Assumption \ref{A_Pi},} \\
&\text{ with } \lambda:[0,T]\times\C^2\rightarrow \C \text{ Borel such that conditions \eqref{CondLambda0}, \eqref{CondLambda1}} \\
&\text{and \eqref{CondLambda3} are fulfilled} \Big\}.
\end{split}
\end{align}

\begin{corollary} \label{ExpAddMgPb}
Suppose that Assumptions \ref{A_PAI} and \ref{A_Pi}  are verified.
Then $(X,S)$ is a solution of the strong martingale problem related to $(\Da, \a, \rho^S)$ where,  for $f \in \Da$ of the form \eqref{DFLaplace}, $\a(f)$ is defined by
\begin{equation}
\label{ExpAddA}
\a(f)(t,x,s) = \int_{\C^2} d \Pi(z_1, z_2) x^{z_1} s^{z_2} \left\lbrace  \dfrac{d\lambda(t,z_1,z_2)}{d\rho^S_t} + \lambda(t,z_1,z_2) \dfrac{d\kappa_t(z_1,z_2)}{d\rho^S_t} \right\rbrace,
\end{equation}
for all $t \in [0,T], x,s >0.$
\end{corollary}
\begin{proof}
By Theorem \ref{T17} notice that $f \in \Da $ defined by \eqref{DFLaplace}
is continuous, which implies that \eqref{D26}  is fulfilled.
By \eqref{CondLambda3}, $\a(f)$ belongs to $\cL$ defined in \eqref{E23} 
and Condition \eqref{aFinite} is fulfilled. 
Finally  \eqref{SMgDecomp} is a consequence of 
\eqref{EMartFubini} in  Theorem \ref{T17}.

\end{proof}

Under additional conditions, one can say more about the martingale decomposition given by the strong martingale problem related to $(\Da, \a, \rho^S)$.

\begin{proposition}
\label{ExpAddSqIntgMg}
Suppose that Assumptions \ref{A_PAI} and \ref{A_Pi}  are verified. Let $f\in  \Da$ as defined in  \eqref{ExpAddADom}. Suppose that the following conditions are verified.
\begin{enumerate}[label=\alph*)]
\item \label{ExpAddSqIntgMg_1} $\displaystyle I_0:=\textit{{\rm Re}(supp }\Pi\text{)} \subset D/2,$
\item \label{ExpAddSqIntgMg_2} $\displaystyle \int_{\C^2}d|\Pi|(z_1,z_2) \int_0^T |\lambda(u,z_1,z_2)|^2 \rho_{du}(z_1,z_2) < \infty.$
\end{enumerate} Then, the martingale $\displaystyle t\mapsto \widehat{M}^\lambda_t=f(t,X_t,S_t) - \int_0^t \a(f)(u,X_{u-},S_{u-}) \rho^S_{du}$ is square integrable.
\end{proposition}
\begin{proof}
Let $t\in[0,T]$ and $\widehat{M}^\lambda$ as defined in \eqref{EMartFubini}, which is a martingale by Theorem \ref{T17}. 
 By \eqref{E153} we have
\begin{equation}
\label{EMMLambda}
\widehat{M}^\lambda_t = \int_{\C^2} d\Pi(z_1,z_2) M^\lambda_t(z_1,z_2),
\end{equation} where $M^\lambda(z_1,z_2)$ was defined in \eqref{MLambda}.
By Lemma \ref{ExpAddLemma}, 
For every $(z_1,z_2)\in D/2$, 
 $M^\lambda(z_1,z_2)$ is a square integrable martingale and  \eqref{MLambda2Moment} holds.
By Fubini theorem, integrating both sides of \eqref{MLambda2Moment} with respect to $\vert \Pi \vert$, gives
\begin{align*}
\begin{split}
\E{\int_{\C^2}d|\Pi|(z_1,z_2) |M^\lambda_t(z_1,z_2)|^2} & =  \int_{\C^2}d|\Pi|(z_1,z_2) \E{ |M^\lambda_t(z_1,z_2)|^2 } \\
& =  \int_{\C^2}d|\Pi|(z_1,z_2) \int_0^t e^{\kappa_u(2{\rm Re}(z_1), 2{\rm Re}(z_2))} |\lambda(u,z_1,z_2)|^2 \rho_{du}(z_1,z_2) \\
&\leq  \sup_{u\in [0,T], (a,b)\in I_0} e^{\kappa_u(a, b)} \int_{\C^2}d|\Pi|(z_1,z_2) \int_0^t |\lambda(u,z_1,z_2)|^2 \rho_{du}(z_1,z_2).
\end{split}
\end{align*}
Now, by point \ref{propAddPt3} of Proposition \ref{propAdd} and condition \ref{ExpAddSqIntgMg_2},
 the right-hand side is finite.
This together with \eqref{EMMLambda} and Cauchy-Schwarz show
that  $\widehat{M}^\lambda$ is  square integrable. 

\end{proof}
\begin{proposition} \label{P24}\
We suppose  the validity of Assumptions \ref{A_PAI}.
\begin{enumerate}[label=\arabic*)]
\item \label{P24_1}
Assumption \ref{E0} is verified. More precisely
\begin{enumerate}[label=\roman*)]
	\item \label{P24_1_i} $id: (t,x,s)\longmapsto s \in \Da$ and
	\begin{equation}\label{ExpAddId}
	\a(id)(t,x,s)=s\dfrac{d\kappa_t(0,1)}{d\rho^S_t},\quad \forall t\in[0,T], x,s>0.
	\end{equation}
	\item \label{P24_1_ii} $(t,x,s)\longmapsto s^2 \in \Da$ and
	\begin{equation}\label{ExpAddAlpha}
	\widetilde\a(id)(t,x,s)= s^2,\quad \forall t\in[0,T], x,s>0.
	\end{equation}
\end{enumerate}
\item \label{P24_2} Let $\Pi$ be a finite signed Borel measure on $\C^2$ verifying Assumption \ref{A_Pi}. 
Let $f\in \Da$ of the form \eqref{ExpAddADom},
 such that $\widetilde f = f\times id \in \Da$. 
Then $\widetilde{\a}$, defined in  \eqref{EAtilde}, 
 is given by, $\forall t\in[0,T], x,s>0$,
\begin{equation}
\label{ExpAddATilde}
\widetilde{\a}(f)(t,x,s) = \int_{\C^2} d \Pi(z_1, z_2) \lambda(t,z_1,z_2) x^{z_1} s^{z_2+1} \dfrac{d\rho_t(z_1, z_2,0,1)}{d\rho^S_t}.
\end{equation}
\end{enumerate}
\end{proposition}

\begin{proof} \

We first address item \ref{P24_1}. 
\begin{enumerate}[label=\roman*)]
\item Let  $\Pi_1(z_1,z_2)=\delta_{\{z_1=0,z_2=1\}}$ and $\lambda\equiv1$.
Since  by Assumption \ref{A_PAI} $(0,1) \in D$,
 $\Pi_1$ fulfills Assumption \ref{A_Pi}.
The other conditions \eqref{CondLambda0},  \eqref{CondLambda1},  \eqref{CondLambda3}
defining $\Da$ in \eqref{ExpAddADom} are trivially satisfied.
Consequently $id \in \Da$ and
\eqref{ExpAddId} follows from \eqref{ExpAddA}.
\item Let $\Pi_2(z_1,z_2)=\delta_{\{z_1=0,z_2=2\}}$ and $\lambda\equiv1$.
Again, by Assumption \ref{A_PAI} $(0,2) \in D$,
 and  $\Pi_2$ fulfills Assumption \ref{A_Pi}.
Similar arguments as for \ref{P24_1_i} allow to show that
 $(t,x,s)\longmapsto s^2 \in \Da$.
\end{enumerate}

Formula \eqref{ExpAddATilde}  constitutes a direct   application of \eqref{ExpAddA},  
taking into account \eqref{ExpAddADom},
  which establishes item \ref{P24_2}.
In particular \eqref{ExpAddAlpha} holds.
\end{proof}

\section{The basic BSDE  and the deterministic problem}
\label{S3_BSDE}

\setcounter{equation}{0}

\subsection{Forward-backward SDE}
\label{SGF}

We consider two $\cF_t$-adapted  processes  $(X,S)$
fulfilling the martingale
 problem related to $(\Da, \a, A)$ stated in Definition \ref{DSMP}
 under  Assumption \ref{E0}. We denote by $M^S$  (resp. $V^S$) the martingale part
(resp. the predictable bounded variation component) 
 of the special semimartingale $S$. 

Let $f: [0,T] \times \cO \times \C^2 \longrightarrow \C$ 
be a locally bounded function 
and a random variable $h := g(X_T, S_T)$, for  some continuous function $g: \cO \rightarrow \C$. In this chapter we concentrate on forward-backward SDEs of the type
\begin{equation}
\label{BSDEMgMarkov}
Y_t = h + \int_t^T f(u,X_{u-}, S_{u-}, Y_{u-},Z_u)d A_u - \int_t^T Z_u dM^S_u - (O_T-O_t), \quad t\in [0,T].
\end{equation}


\begin{definition} 
\label{DefBSDE}
A triplet $(Y,Z,O)$ of processes  is called {\bf solution} of \eqref{BSDEMgMarkov} if the following
 conditions hold. 
\begin{enumerate}[label=\arabic*)]
\item \label{DefBSDE_1} $(Y_t)$ is $\cF_t$-adapted;
\item \label{DefBSDE_2} $(Z_t)$ is $\cF_t$-predictable and
	\begin{enumerate}
	\item $\int_0^T |Z_u|^2 d\langle M^S\rangle_u < \infty$ a.s.
	\item $\int_0^T |f(u,X_{u-}, S_{u-}, Y_{u-},Z_u)| d \norm{A}_u < \infty $ a.s.
	\end{enumerate}
\item \label{DefBSDE_3} Equality \eqref{BSDEMgMarkov} holds and
 $(O_t)$ is an $\cF_t$-local martingale such that $\langle O,M^S\rangle=0$ and $O_0 = 0$ a.s.
\end{enumerate}
\end{definition}




Our object of study is the formulation of a deterministic problem linked to the BDSE \eqref{BSDEMgMarkov}, generalizing the
 "classical" semilinear PDE in the Brownian motion case.
 In particular we look for solutions $(Y,Z,O)$ for which there is a function $y \in \Da$
such that $\widetilde y = y \times id \in \Da$ and a
 locally bounded Borel function 
$
 z: [0,T]\times \cO \longrightarrow \C,
$ such that 
\begin{equation}
\label{solForm}
Y_t = y(t,X_t,S_t), \quad
Z_t = z(t,X_{t-},S_{t-}), \quad \forall t \in [0,T],
\end{equation} and
\begin{equation}
\label{condIntegYZ}
\int_0^t \vert Z_u\vert^2 d\langle M^S\rangle_u  < \infty \text{ a.s.} \quad    
\int_0^t |f(u, X_{u-}, S_{u-}, Y_{u-}, Z_u)|d \norm{A}_u <\infty \text{ a.s.}
\end{equation}
By \eqref{solForm} and \eqref{condIntegYZ}, 
Conditions \ref{DefBSDE_1} and \ref{DefBSDE_2} of Definition
\ref{DefBSDE} are obviously fulfilled.
Consequently the triplet $(Y,Z,O)$ where
\begin{equation}
\label{orthMG}
O_t := Y_t-Y_0-\int_0^t Z_u dM^S_u + \int_0^t f(u,X_{u-},S_{u-},Y_{u-},Z_u) dA_u,
\end{equation}
is a solution of \eqref{BSDEI} provided that
\begin{eqnarray}
&1.&(O_t) \text{ is an } \cF_t\text{-local martingale}, \label{cond1} \\
&2.&\langle O,M^S\rangle=0, \label{cond2} \\
&3.& Y_T = g(X_T,S_T). \label{cond3}
\end{eqnarray}
Since $(X,S)$ solves the strong martingale problem related to  $(\Da, \a, A)$,
replacing \eqref{solForm} in expression \eqref{orthMG},  
Condition \eqref{cond1} can be rewritten saying that
$
\int_0^t \a(y)(u,X_{u-}, S_{u-}) dA_u + \int_0^t f(u,X_{u-},S_{u-},Y_{u-},Z_u)dA_u
$ is an $\cF_t$-local martingale. This implies that
\begin{equation}
\label{cond1bis}
\int_0^t \a(y)(u,X_{u-}, S_{u-}) dA_u + \int_0^t f(u,X_{u-},S_{u-},Y_{u-},Z_u)dA_u = 0.
\end{equation}
On the other hand, Condition \eqref{cond2} implies 
\begin{equation}
\label{cond2bis}
\langle M^Y, M^S\rangle_t = \int_0^t Z_u d\langle M^S\rangle_u,
\end{equation} where $M^Y$ denotes the martingale part of $Y$. 
By Lemma \ref{LC} and item \ref{RSpecial_2} of Corollary \ref{RSpecial}, Condition \eqref{cond2bis} can be re-expressed as
\begin{equation}
\label{cond2ter}
\int_0^t \widetilde\a(y)(u,X_{u-},S_{u-})dA_u  = \int_0^t z(u,X_{u-},S_{u-}) \widetilde\a(id)(u,X_{u-},S_{u-})dA_u.
\end{equation}
Condition \eqref{cond3} requires $y(T,\cdot,\cdot) = g(\cdot,\cdot)$.
This allows to state below a representation theorem.

\begin{theorem} \label{T19}
Suppose the existence of a 
function $y,$ such that $y, \widetilde{y}:=y\times id$ belong to $\Da$, 
and a Borel locally bounded   function $z$, solving the system
\begin{eqnarray}
\a(y)(t, x, s) &=& - f(t, x, s, y(t, x, s), z(t, x, s))\label{suffConds1}\\
\widetilde{\a}(y)(t, x, s) &=& z(t, x, s) \widetilde{\a}(id)(t, x, s) \label{suffConds2},
\end{eqnarray}for $t\in [0,T]$ and $(x,s)\in\cO$, where the equalities hold
 in $\cL$, with the terminal condition $y(T,.,.) = g(.,.)$.

Then the triplet $(Y,Z,O)$ defined by 
\begin{equation} \label{ESolYZ}
 Y_t = y(t,X_t,S_t), \; Z_t = z(t,X_{t-},S_{t-})
\end{equation}
 and $(O_t)$ 
 given by \eqref{orthMG}, is a solution to the BSDE \eqref{BSDEMgMarkov}.
\end{theorem}

\begin{proof} 
The triplet $(Y,Z,O)$ fulfills the three conditions of
 Definition \ref{DefBSDE}
provided that \eqref{condIntegYZ} is verified.
Indeed, since $y  \in \Da$ then 
the  integral $\int_0^t |f(u,X_{u-},S_{u-},Y_{u-},Z_u)| d\norm{A}_u,$ 
is finite taking into account \eqref{aFinite}. 
Since $z$ is locally bounded, then 
$\int_0^T \vert Z_u\vert^2 d\langle M\rangle_u$ is also finite.
This concludes the proof of the theorem.
\end{proof}
\begin{remark}\label{RC33}\
\begin{enumerate}[label=\arabic*.]
\item \label{RC33_1}
The statement of Theorem \ref{T19} can be generalized relaxing the assumption on $z$ to be locally bounded. 
We replace this with the condition
\begin{equation}
\label{ERC43}
\int_0^T z^2(u,X_{u-},S_{u-})\widetilde\a(id)(u,X_{u-},S_{u-}) dA_u < \infty \text{ a.s.}
\end{equation}
This is equivalent  to $\int_0^T \vert Z_u\vert ^2
 d\langle M \rangle_u < \infty  \text{ a.s.}$

\item \label{RC33_2} In particular, if $z$ is locally bounded a.s., then 
\eqref{ERC43} is fulfilled.
\end{enumerate}
\end{remark}
\begin{remark}\label{RUniqueness}
Theorem \ref{T19} constitutes also an existence theorem for particular BSDEs. 
If $M^S$ is a square integrable martingale and the function 
 $\hat f$ associated with $f$, 
fulfills some Lipschitz type conditions then the solution $(Y,Z,O)$
provided by \eqref{ESolYZ} is unique  in the class of processes
introduced in   \cite[Theorem 3.1]{CarFerSan08}. 
\end{remark}
The presence of the local martingale $O$ is closely related to the classical martingale representation property. In fact, if $(\Omega, \cF, \P)$ verifies the local martingale representation property with respect to $M^S$, then $O$ vanishes.
\begin{proposition}
Suppose that $(\Omega, \cF, \P)$ fulfills the local martingale representation property with respect to $M$. Then, if $(Y,Z,O)$ is a solution to \eqref{BSDEMgMarkov} in the sense of Definition \ref{DefBSDE}, then, necessarily $O_t=0, \; \forall t\in [0,T]$.
\end{proposition}
\begin{proof}
Since $(O_t)$ is an $\cF_t$-local martingale, there is a predictable process
 $(U_t)$ such that
$
O_t=O_0+\int_0^t U_u dM^S_u, \; \forall t \in [0,T], \text{ with } O_0=0.
$
So the condition $\langle O,M^S\rangle \equiv 0$ implies $\int_0^. U_u d \langle M^S\rangle_u=0.$

Consequently, $\;U \equiv 0 \; d \P \otimes d\langle M^S\rangle \; \text{a.e.,}\;$ and so $\;O_t=O_0=0 \; \forall t \in [0,T].$

\end{proof}

\begin{remark} \label{RLink}
We end this section recalling that  the forward-backward  SDE 
\eqref{BSDEMgMarkov} that we study is a particular of  the general
BSDE 
\eqref{BSDEI} driven by a martingale,  provided in the Introduction.
The link is given by \eqref{EFormTildef}.
\end{remark}

\subsection{Illustration 1: the Markov semigroup case}

Let us consider the case of Section \ref{SMarkov}  with related notations. Let $S = X^{0,x}$ be a solution of the strong martingale problem related to
 $(\Da, \a, A)$, see Definition \ref{DSMPMod}.
Let $(P_t)$ be the semigroup introduced in \eqref{EPT}, fulfilling Assumption \ref{A_Markov} with generator $L$
defined in Definition \ref{DGen}. Let $f: [0,T] \times \R \times \C \longrightarrow \C$ be a locally bounded function and a continuous function $g: \R \longrightarrow \C$.\\
Here we have of course $S = M^S + V^S$ where $V^S = \int_0^\cdot \a(id)(u,S_{u-}) du$ and $id(s) \equiv s$.

Theorem \ref{T19} gives rise to the following proposition.

\begin{proposition}
Suppose the existence of  a function $y:[0,T]\times \R \to \C$ and a Borel locally bounded function $z:[0,T]\times \R \to \C$ fulfilling the following conditions.

\begin{enumerate}[label=\roman*)]
\item $t\mapsto y(t,\cdot)\; (\text{resp. } \tilde y(t,\cdot))$  takes value in $D(L)$ and it is continuous with respect to the graph norm.
\item $t\mapsto y(t,\cdot)\; (\text{resp. }\tilde y(t,\cdot))$ is of class $C^1$ with values in $E$.

\item For $(t,x) \in [0,T] \times \R$,
\begin{eqnarray*}
 \partial_t y(t,x) + Ly(t,\cdot)(x)  &=&  - f(t,x,y(t,x),z(t,x)),   \\
z(t,\cdot)  \widetilde L(id) &=&  \widetilde L y(t,\cdot),  \\
y(T,.) &=& g,
\end{eqnarray*}
where $\widetilde L \varphi = L \tilde \varphi - \varphi L id  - id L \varphi$.
\end{enumerate}
Then the triplet
 $(Y,Z,O)$, where 
$$ Y_t = y(t,X_t,S_t), \; Z_t = z(t,X_t,S_t),$$ and $(O_t)$ is given by \eqref{orthMG} is 
a solution to the BSDE \eqref{BSDEMgMarkov}.
\end{proposition}

\begin{remark} \label{RMSC}
If $S = \sigma W$ with $\sigma > 0$
and $\varphi:[0,T] \times \R \rightarrow \C$ is of class $C^{1,2}$  
 then $\a(\varphi) = \partial_t \varphi + \frac{\sigma^2}{2} \partial_{ss} \varphi $ and
$\widetilde \a(\varphi)   = \sigma^2 \partial_s \varphi = \widetilde \a(id) \partial_s \varphi $. 
In the case where $L$ is a generic generator, 
 the formal quotient $\frac{\widetilde \a(\varphi)}{\widetilde \a(id)}$ 
can be considered as a sort of generalized derivative.
\end{remark}

\subsection{Illustration 2: the diffusion case}
Consider the case of where $(X,S)$ is diffusion process as given in equations \eqref{expleDiff}. 
We recall that in that case, the operator $\a$, for $\varphi\in \mathcal{C}^{1,2}([0,T]\times \R^2)$, 
is given by
\begin{eqnarray*}
\a(\varphi) &=& \partial_t \varphi + b_S \partial_s \varphi + b_X \partial_x \varphi\\
&+&  \frac{1}{2} \left\lbrace |\sigma_S|^2\partial_{ss} \varphi + |\sigma_X|^2 \partial_{xx} \varphi + 2 \langle \sigma_S, \sigma_X\rangle \partial_{sx} \varphi  \right\rbrace .
\end{eqnarray*}

\begin{corollary}
Let $(y,z)$ be a solution of the PDE
\begin{eqnarray}
\a(y)(t,x,s) &=& -f(t,x,s,y(t,x,s),z(t,x,s))\\
|\sigma_S|^2 z(t,x,s) &=& |\sigma_S|^2 \partial_s y(t,x,s) + \langle \sigma_S, \sigma_X\rangle \partial_x y(t,x,s),
\end{eqnarray} with terminal condition $y(T,.,.) = g(.,.)$. Then the triplet
 $(Y,Z,O)$, where 
$$ Y_t = y(t,X_t,S_t), \; Z_t = z(t,X_t,S_t),$$ and $(O_t)$ is given by \eqref{orthMG} is 
a solution to the BSDE \eqref{BSDEMgMarkov}.
\end{corollary}




\section{Explicit solution for F\"ollmer-Schweizer decomposition in the basis risk context}
\label{S4_BSDE}

\setcounter{equation}{0}
\subsection{General considerations}

We will discuss in this section the important 
 F\"ollmer-Schweizer decomposition, denoted shortly F-S decomposition. It is a generalization of the well-known Galtchouk-Kunita-Watanabe decomposition
 for martingales, to the more general case of semimartingales.
Our task  will consist in providing explicit expressions for the F-S decomposition
in several situations.
Let $S$ be a special semimartingale  with
 canonical decomposition $S=M^S+V^S$.
In the sequel we will convene that the space $L^2(M^S)$ consists of the predictable processes $(Z_t)_{t \in [0,T]}$ such that $\E{\int_0^T | Z_u|^2d\langle M^S \rangle_u} < \infty$ and $L^2(V^S)$ will denote the set of all predictable processes $(Z_t)_{t \in [0,T]}$ such that $\E{\left(\int_0^T |Z_u|d\| V^S \|_u\right)^2} < \infty$. The intersection of these two spaces is denoted 
\begin{equation}\label{ETheta}
\Theta:= L^2(M^S) \cap L^2(V^S).
\end{equation}

The F\"ollmer-Schweizer decomposition is defined as follows.
\begin{definition}
\label{FSDefinition}
Let $h$ be a (possibly complex valued)  square integrable $\cF_T$-measurable random variable. We say that $h$ admits an F-S decomposition
 with respect to $S$ if it can be written as
\begin{equation}
\label{FSDec}
h = h_0 + \int_{0}^{T} Z_u d S_u + O_T, \P-a.s.,
\end{equation}
where $h_0$ is an $\cF_0$-measurable r.v., $Z \in \Theta$ and $O = (O_t)_{t\in[0,T]}$
 is a square integrable martingale, strongly orthogonal to $M^S$.
\end{definition}
\begin{remark}\label{RFS}\
\begin{enumerate}[label=\arabic*)]
\item \label{RFS_1} The notion of weak and strong orthogonality is discussed for instance in 
\cite[Section 4.3]{ProtterBook} and \cite[Section 1.4b]{JacodBook}.
 Let $L$ and $N$ be two $\cF_t$-local martingales,
 with null initial value. 
 $L$ and $N$ are said to be {\it strongly orthogonal} if  $LN$ is a local martingale.
 If $L$ and $N$ are locally square integrable, then they are strongly orthogonal if and only if $\langle L,N\rangle=0$.
The definition of locally square integrable martingale is given for instance
just before \cite[Theorem 49 in Chapter 1]{ProtterBook}.
\item \label{RFS_2} The F-S decomposition makes also sense for complex valued square integrable random variable $h$. In that case the triplet $(h_0,Z,O)$ is generally complex.
\item \label{RFS_3} If $h$ admits an F-S decomposition  \eqref{FSDec} then
the complex conjugate   $\bar{h}$ admits an F-S decomposition given by 
\begin{equation}
\label{FSDecBar}
\bar{h} = \bar{h}_0 + \int_{0}^{T} \bar{Z}_u d S_u + \bar{O}_T, \P-a.s.
\end{equation}
\end{enumerate}
\end{remark}

The F-S decomposition has been extensively studied in the literature: sufficient 
conditions on the process $S$ were given so that every square integrable random variable has such a decomposition. A well-known condition 
ensuring the existence of such a decomposition is  
 the so called \textbf{structure condition} (SC).

\begin{definition}
\label{SC}
We say that a special semimartingale $S=V^S+M^S$ satisfies the \textbf{structure condition} (SC) if 
there exists a predictable process $ \alpha$
  such that 
\begin{enumerate}
\item $V^S_t=\int_0^t  \alpha_u d \langle M^S \rangle_u$,
\item $\int_0^T \alpha_u^2 d \langle M^S \rangle_u <  \infty$ a.s.
\end{enumerate}
\end{definition}
The latter quantity plays a central role in the F-S decomposition.
The associated process  
\begin{equation}
\label{MeanVarTradOff}
K_t := \int_0^t \alpha_u^2 d \langle M^S \rangle_u \text{ for } t \in [0,T],
\end{equation}
is called \textbf{mean variance trade-off} process.

\begin{remark}\label{RMS}\
\cite{monat} proved that, under  (SC) and the additional condition that the
 process $K$ is uniformly bounded, the F-S decomposition of any real valued square integrable random variable exists and it is unique.
More recent papers about the subject are \cite{Schweizer2001}, \cite{Cerny} and references therein. 
\end{remark}

This general decomposition  refers to the process $S$ as underlying and
it will be applied in the context of mean-variance hedging
under basis risk, where  $X$ is an observable price process 
 of a non-traded asset. \\
As in  previous sections, we consider a couple $(X,S)$ 
  verifying the martingale problem \eqref{pbMg}, and we suppose Assumption \ref{E0} 
 to be fulfilled. 
In the sequel we do not necessarily assume (SC) for $S$.
\begin{definition} \label{DWeakFS}
Let $h$ be a square integrable $\cF_T$-measurable random variable. We say that $h$ admits a {\bf weak F-S decomposition} with respect to $S$ if it can be written as
\begin{equation}
\label{FSDecWeak}
h = h_0 + \int_{0}^{T} Z_u d S_u + O_T, \P{\rm-a.s.},
\end{equation}
where $h_0$ is an $\cF_0$-measurable r.v., 
$Z$ is a predictable process such that $\int_0^T | Z_u|^2d\langle M^S \rangle_u < \infty$
a.s., $\int_0^T | Z_u|d\Vert V^S \Vert_u < \infty$ a.s.
and $O$ is a local martingale such that $\langle O,M^S \rangle = 0$
with $O_0 = 0$.
\end{definition}
Finding a weak F-S decomposition \eqref{FSDecWeak} $(h_0,Z,O)$
for some r.v. $h$ is equivalent to finding a solution $(Y,Z,O)$
 of 
the BSDE
\begin{equation}
\label{FSasBSDE}
Y_t = h - \int_{t}^{T} Z_u d S_u - (O_T-O_t). 
\end{equation}
The link is given by $Y_0 = h_0$.
Equation \eqref{FSasBSDE} can be seen as a special case of BSDE \eqref{BSDEMgMarkov},
 where the driver $f$ is linear in $z$, of the form 
\begin{equation} \label{BSpecialf}
f(t, x, s, y, z) = -\a(id)(t, x, s) z. 
\end{equation}
This point of view was taken for instance by \cite{FS1994}.

\begin{remark}\label{FS-BSDE}
Let $(Y,Z,O)$ be a solution of \eqref{FSasBSDE} with
$Z \in \Theta$, where $\Theta$ has been defined in \eqref{ETheta}
and $O$ is a square integrable martingale. Then $h$
admits an F-S decomposition \eqref{FSDec} with $Y_0 = h_0$.

\end{remark}

We consider the case of the final value $h = g(X_T,S_T)$ for some continuous function $g$.
Theorem \ref{T19} can be applied to obtain the result below.
\begin{corollary}
\label{FSConditions}
Let $y$  (resp. $z$): $[0,T] \times \cO \rightarrow \C$.
We suppose that the following conditions are verified.
\begin{enumerate}[label=\arabic*)]
\item \label{FSConditions1} $y, \widetilde{y}:=y\times id$ belong to $\Da$. 
\item \label{FSConditions2}
$z$ 
verifies \eqref{ERC43} of Remark \ref{RC33}.
\item  \label{FSConditions3}
 $(y,z)$ solve the problem
\begin{eqnarray} 
\a(y)(t, x, s) &=& \a(id)(t, x, s) z(t, x, s),  \label{FSConds1} \\
\widetilde{\a}(y)(t, x, s) &=& \widetilde\a(id)(t, x, s) z(t, x, s),
\label{FSConds2}
\end{eqnarray}
where the equalities hold in $\cL$, 
with the terminal condition $y(T,.,.) = g(.,.)$.
\end{enumerate}
 Then the triplet $(Y,Z,O)$, where $$ Y_t = y(t,X_{t},S_{t}), \; Z_t = z(t,X_{t-},S_{t-}), \; O_t=Y_t - Y_0 - \int_{0}^{t} Z_u d S_u,$$ is a solution to the linear BSDE \eqref{FSasBSDE} linked to the weak F-S decomposition.
\end{corollary}

\begin{remark}
\label{RBSDEFS}
We recall that, setting $h_0=y(0, X_0, S_0)$, the triplet $(h_0,Z,O)$ is a candidate for a true F-S decomposition, see Definition \ref{FSDefinition}. Sufficient conditions for this are stated below.

\begin{enumerate}[label=\alph*)]
\item \label{RBSDEFS_1} $h=g(X_T,S_T) \in L^2(\Omega).$
\item \label{RBSDEFS_2} $(z(t, X_{t-}, S_{t-})) \in \Theta$ i.e.
	\begin{itemize}
	\item $\E{\int_0^T \left|z(t, X_{t-}, S_{t-})\right|^2 \widetilde\a(id)(t,X_{t-}, S_{t-}) 
	dA_t } < \infty$.
	\item $\E{\left(\int_0^T \left|z(t, X_{t-}, S_{t-})\right| 
	\Vert \a(id)(t, X_{t-}, S_{t-}) dA \Vert_t \right)^2} < \infty$.
	\end{itemize}
\item \label{RBSDEFS_3} $\left(y(t,X_t,S_t) - \int_0^t \a(y)(u,X_{u-}, S_{u-}) dA_u\right)$ is an $\cF_t$-square integrable martingale. 
\end{enumerate}
\end{remark}
We remark that \ref{RBSDEFS_2} and \ref{RBSDEFS_3} imply by additivity that $O$ is a square integrable martingale. In fact, $\forall t\in[0,T]$
\begin{equation}
\label{E_RBSDEFS}
O_t = y(t,X_t,S_t) -y(0,X_0,S_0)- \int_0^t \a(y)(u,X_{u-}, S_{u-}) dA_u - \int_0^t z(u,X_{u-}, S_{u-}) dM^S_u.
\end{equation}
\subsection{Application: exponential of additive processes}
\label{PAI}

We will investigate in this section a significant context where the equations in Corollary
 \ref{FSConditions} can be solved, yielding the weak F-S decomposition
and we can give sufficient conditions so that the true F-S decomposition is fulfilled.
 We focus on exponential of additive processes.
Another example will be given in Section \ref{S412}.
Let $(X,S)$ be a couple of exponential of semimartingale additive processes, as introduced in Section \ref{expleExpAdd}.

\begin{proposition}\label{PSC}
Under Assumption \ref{A_PAI}, $S$ verifies the (SC) condition given in Definition \ref{SC} if and only if
\begin{equation}
\int_0^T \left(\frac{d\kappa_t(0,1)}{d\rho^S_t}\right)^2 d\rho^S_t <\infty
\end{equation}

In this case, the mean variance trade-off process $K$ is deterministic and given by
\begin{equation}
\label{MVT_PAI}
K_t = \int_0^t \left(\frac{d\kappa_u(0,1)}{d\rho^S_u}\right)^2 d\rho^S_u <\infty, \;\forall t\in[0,T].
\end{equation}
\end{proposition}
\begin{proof}
It follows from Corollary \ref{CCS} and item \ref{propAddPt5} of Proposition  \ref{propAdd}.
\end{proof}

We look for the F-S decomposition of an $\cF_T$-measurable random variable $h$ of the form $h:=g(X_T,S_T)$ for a function $g$ such that
\begin{equation} \label{Eh} 
 g(x,s) = \int_{\C^2} d \Pi(z_1, z_2) x^{z_1} s^{z_2},
\end{equation}
 where $\Pi$ is finite Borel complex measure. 
\begin{remark}
This family of random variables includes many examples as,  for example, the call and put options payoffs. Indeed, we have, for $K,s>0$ and arbitrary $0<R<1$,
 $$
 (s-K)^+-s = \frac{1}{2\pi i} \int_{R-i\infty}^{R+i\infty}s^z \frac{K^{1-z}}{z(z-1)}dz.
 $$
 Moreover, for arbitrary $U>0$,
  $$
 (K-s)^+ = \frac{1}{2\pi i} \int_{U-i\infty}^{U+i\infty}s^z \frac{K^{1-z}}{z(z-1)}dz.
 $$
 For more details, see for example \cite{Hubalek2006}, \cite{EberleinGlau2010} and \cite{gor2013variance}.
\end{remark}

In Section \ref{expleExpAdd}, Corollary \ref{ExpAddMgPb} states that $(X,S)$ fulfills the martingale problem with respect to $(\Da, \a, \rho^S)$ where the objects
 $\Da, \a$ and $\rho^S$ were introduced respectively in \eqref{ExpAddADom}, \eqref{ExpAddA}, \eqref{ExpAddRhoS}.
In order to determine the F-S decomposition (in its weak form given in \eqref{DWeakFS})
we make use of Corollary \ref{FSConditions}.
We look for a function $y$  (resp. $z$): $[0,T] \times \R^2 \rightarrow \C$ 
such that Hypotheses \ref{FSConditions1}, \ref{FSConditions2} and \ref{FSConditions3}
are fulfilled. In agreement with  definition of $\Da$  given in \eqref{ExpAddADom}
 we select  $y$  of the form
\begin{equation}\label{Eh1}
 y(t,x,s) = \int_{\C^2} d \Pi(z_1, z_2)x^{z_1} s^{z_2} \lambda(t,z_1,z_2),
\end{equation}
where $\Pi$ being the same finite complex measure as in \eqref{Eh}
and $\lambda:[0,T] \times \C^2 \rightarrow \C$.
We will start by writing ''necessary'' conditions for a couple $(y,z)$, 
such that $y$ has the form \eqref{Eh1}, to be
solutions of  \eqref{FSConds1}.

Suppose that the couple $(y,z)$ fulfills  \eqref{FSConds1} and \eqref{FSConds2} of Corollary \ref{FSConditions}.
We consider the expressions of $ \a(id), \widetilde{\a}(id)$ 
given by \eqref{ExpAddId}, \eqref{ExpAddAlpha},
and $\a(y),  \widetilde{\a}(y)$
 given by \eqref{ExpAddA} and \eqref{ExpAddATilde}, for $f = y$.
We replace them in the two above mentioned conditions \eqref{FSConds1} and \eqref{FSConds2}  
 to obtain the following equations for $\lambda$
($d \rho^S_t$ a.e.).

\begin{equation}
\label{E450}
\begin{aligned}
\int_{\C^2} d \Pi(z_1, z_2) x^{z_1} s^{z_2} \left\lbrace  \dfrac{d\lambda(t,z_1,z_2)}{d\rho^S_t} + \lambda(t,z_1,z_2) \dfrac{d\kappa_t(z_1,z_2)}{d\rho^S_t} \right\rbrace &=& s \dfrac{d\kappa_t(0,1)}{d\rho^S_t} z(t, x, s) 
\\
\int_{\C^2} d \Pi(z_1, z_2) \lambda(t,z_1,z_2) x^{z_1} s^{z_2+1} \dfrac{d\rho_t(z_1, z_2,0,1)}{d\rho^S_t} &=& s^2 z(t, x, s).
\end{aligned}
\end{equation}
The final condition $y(T,\cdot, \cdot) = g(\cdot,\cdot)$ produces
\begin{equation} \label{E450bis}
\int_{\C^2} d \Pi(z_1, z_2) x^{z_1} s^{z_2} \lambda(T,z_1,z_2) = \int_{\C^2} d \Pi(z_1, z_2) x^{z_1} s^{z_2}.
\end{equation}
Replacing $z$ from the second line of \eqref{E450} in the first line, 
by identification of the inverse Fourier-Laplace transform, it follows
that $\lambda$ verifies 
\begin{eqnarray}
\label{lambdaODE}
\dfrac{d\lambda(t,z_1,z_2)}{d\rho^S_t} &=& \lambda(t,z_1,z_2) \left\lbrace  \dfrac{d\kappa_t(0,1)}{d\rho^S_t}\dfrac{d\rho_t(z_1, z_2,0,1)}{d\rho^S_t}  - \dfrac{d\kappa_t(z_1,z_2)}{d\rho^S_t}\right\rbrace \label{lambdaODE1}
 \\
\lambda(T,z_1,z_2) &=& 1,  \label{lambdaODE2}
\end{eqnarray}
for all $(z_1,z_2)\in supp\;\Pi$.
Without restriction of generality we 
can clearly set $\lambda(\cdot, z_1,z_2) = 0$ for
 $(z_1,z_2) $ outside the support of $\Pi$.
We observe that for fixed $z_1, z_2$,  \eqref{lambdaODE}
constitutes an  ordinary differential equation
(in the Lebesgue-Stieltjes sense) in time $t$.

We solve now the linear differential equation  \eqref{lambdaODE}. 
Provided that
\begin{equation} \label{EIntegrability}
 u \mapsto \dfrac{d\rho_u(z_1, z_2,0,1)}{d\rho^S_u} \dfrac{d\kappa_u(0,1)}{d\rho^S_u} \in L^1([0,T], d\rho_s),
\end{equation}
the (unique) solution of \eqref{lambdaODE},
is given by
\begin{eqnarray}
\label{E4600}
\lambda(t,z_1,z_2) &=& \exp \left(  \int_t^T \left[ \dfrac{d\kappa_u(z_1,z_2)}{d\rho^S_u} - \dfrac{d\rho_u(z_1, z_2,0,1)}{d\rho^S_u} \dfrac{d\kappa_u(0,1)}{d\rho^S_u} \right] d\rho^S_u \right) \nonumber\\
 &=& \exp \left(  \int_t^T \kappa_{du}(z_1,z_2) - \dfrac{d\rho_u(z_1, z_2,0,1)}{d\rho^S_u} \kappa_{du}(0,1) \right) \\
 &=& \exp \left(  \int_t^T \eta(z_1,z_2, du) \right)\nonumber ,
\end{eqnarray}
where 
\begin{equation}
\label{eta}
\eta(z_1,z_2, t) := \kappa_{t}(z_1,z_2) - \int_0^t\dfrac{d\rho_u(z_1, z_2,0,1)}{d\rho^S_u} \kappa_{du}(0,1),
\end{equation}
which is clearly absolutely continuous with respect to $d \rho^S$. \\
At this point, we have an explicit form of $\lambda$ defining
 the function $y$ intervening in the weak F-S decomposition.
In the sequel we will show that such a choice of $\lambda$ will constitute a sufficient condition
so that $(y,z)$ where $y$ is defined by \eqref{Eh1} and $z$ determined by the second line of \eqref{E450},
is a solution of the deterministic problem given by \eqref{FSConds1} and \eqref{FSConds2}.
In order to check \eqref{EIntegrability} and the validity of \eqref{E450}
and \eqref{E450bis},
we formulate an hypothesis which
reinforces Assumptions  \ref{A_PAI} and \ref{A_Pi}.
\begin{assumption}
\label{setD}
Recall $I_0:=\textit{{\rm Re}(supp }\Pi\text{)}  (\subset \R^2),$
where we convene that  ${\rm Re}(z_1,z_2) = ({\rm Re}(z_1), {\rm Re}(z_2)). $
We denote $I:= 2I_0\cup \{(0,1)\}$ and $\cD$ the set
\begin{equation}
\cD = \Big\{z\in D, \; \int_0^T \left| \frac{d\kappa_u(z_1,z_2)}{d \rho^S_u}\right|^2 d\rho^S_u < \infty \Big\}.
\end{equation} We assume the validity of the properties below.

\begin{enumerate}[label=\arabic*)]
\item $\rho^S$ is strictly increasing.  \label{setD0} 
\item $I_0$ is bounded.
 \label{setD1}
\item $\forall z\in supp\; \Pi,\; z, z+(0,1) \in \cD$ \label{setD2}. 
\item $\displaystyle\sup_{x\in I} \norm{\frac{d(\kappa_t(x))}{d\rho^S_t}}_{\infty} < \infty$. \label{setD3}
\end{enumerate}
\end{assumption}

\begin{remark}  \label{R411} \
\begin{enumerate}[label=\arabic*)]
\item \label{R411Item1}  Assumptions \ref{A_PAI} and
\ref{A_Pi} are  consequences of Assumption  \ref{setD}. 
\item \label{R411Item2} Taking into account Remark \ref{R226}, we emphasize that,
 for the rest of this section,
the statements would not change if we consider that 
 the quantities integrated with respect to the measure $\Pi$
 are null outside its support.
\item \label{R411Item3} $I \subset \cD$, in particular $(0,1)  \in \cD$ 
because of item \ref{setD3} of Assumption \ref{setD}.
\item \label{R411Item4} By previous item and Proposition \ref{PSC}, 
 $S$ verifies the (SC) condition and the mean variance trade-off process $K$ given by \eqref{MVT_PAI}
is deterministic.
\item \label{R411Item5} $I_0 \subset D/2$   (i.e. $supp\;\Pi \subset D/2$).
This follows again by item \ref{setD3} of Assumption \ref{setD}.
\end{enumerate}
\end{remark}

In the sequel we will introduce a useful  notation.
\begin{equation}\label{Egamma}
\gamma_t(z_1,z_2) := \frac{d\rho_t(z_1,z_2,0,1)}{d\rho^S_t}, \; \forall (z_1,z_2)\in D/2,  t\in[0,T].
\end{equation}

Similarly to \cite[Lemma 3.28]{gor2013variance}, we can show the upper bounds below.
\begin{lemma}\ 
\label{LemmaEta}
Under Assumption \ref{setD}, we have the  properties below hold.
\begin{enumerate}[label=\arabic*)]
\item \label{lambdaWellDefined} Condition \eqref{EIntegrability} is verified for  $t\in[0,T], (z_1,z_2)\in supp\; \Pi$. 
\item \label{pt2LemmaEta} There is  a positive constant $c_1$, such that $d\rho^S_s$ a.e.
 $\displaystyle \sup_{(z_1,z_2)\in I_0+i\R^2} \frac{d {\rm Re}(\eta(z_1,z_2,t))}{d\rho^S_t} \leq c_1$.
\item \label{pt3LemmaEta} There are positive constants $c_2$, $c_3$ such that, $d\rho_s$ a.e.
the following property holds. \\
For any $\displaystyle (z_1,z_2)\in I_0+i\R^2, \; 
\bigg| \gamma_t(z_1,z_2) \bigg|^2 \leq \frac{d\rho_t(z_1,z_2)}{d\rho^S_t} \leq c_2 - c_3 \frac{d {\rm Re}(\eta(z_1,z_2,t))}{d\rho^S_t}$.

\item \label{pt4LemmaEta} $\displaystyle \sup_{(z_1,z_2)\in I_0+i\R^2} - \int_0^T 2 {\rm Re}(\eta(z_1,z_2,dt))\exp \left( \int_t^T 2 {\rm Re}(\eta(z_1,z_2,du)) \right) < \infty$. 
\end{enumerate}
\end{lemma}

\begin{proof} 
For illustration we prove item \ref{lambdaWellDefined}, the other points can be shown
by similar techniques as in \cite[Lemma 3.28]{gor2013variance}. \\
Let $t\in[0,T], (z_1,z_2)\in supp\; \Pi$. Condition 
\eqref{EIntegrability} is valid
since $(0,1)\in \cD$, $z, z+(0,1) \in \cD$ and 
$$\left(\int_0^t\left|\dfrac{d\rho_u(z_1, z_2,0,1)}{d\rho^S_u}
\frac{d\kappa_{u}(0,1)}{d\rho^S_u} \right| \rho^S_{du}\right)^2 \leq 
\int_0^t\left|\frac{d\rho_u(z_1, z_2,0,1)}{d\rho^S_u}\right|^2 \rho^S_{du}
\int_0^t\left|\frac{d\kappa_{u}(0,1)}{d\rho^S_u} \right|^2 \rho^S_{du}.$$
\end{proof}
Now, we can state a proposition that gives indeed the 
weak F-S decomposition of  a random variable $h= g(X_T,S_T)$.

\begin{proposition}
\label{FSExpAdd}
We suppose the validity of Assumption
 \ref{setD}.
Let $\lambda$ be defined as
\begin{equation}
\label{lambda}
\lambda(t,z_1,z_2) = \exp \left(  \int_t^T \eta(z_1,z_2, du) \right),
\forall (z_1,z_2) \in D/2,
\end{equation}
where $\eta$ has been defined at \eqref{eta}.
 Then $(Y,Z,O)$ is a solution of the BSDE \eqref{FSasBSDE}, where 
\begin{eqnarray*}
Y_t &=& \int_{\C^2} d \Pi(z_1, z_2) X_{t}^{z_1} S_{t}^{z_2} \lambda(t,z_1,z_2) \\
Z_t &=& \int_{\C^2} d \Pi(z_1, z_2) X_{t-}^{z_1} S_{t-}^{z_2-1} \lambda(t,z_1,z_2) \gamma_t(z_1,z_2) \\
O_t &=& Y_t - Y_0 - \int_{0}^{t} Z_u d S_u,
\end{eqnarray*} 
recalling that $\gamma$ has been defined in \eqref{Egamma}.
\end{proposition}

\begin{proof}
The result will follow from Corollary \ref{FSConditions}
for which we need to check the assumptions. \\
First we prove that the function $y$ defined by \eqref{Eh1},
 where $\lambda$ is defined in \eqref{lambda}, is indeed an element of $\Da$.
Secondly, we prove that the associated $\widetilde y$ also belongs to $\Da$.
Third, we check Condition \eqref{ERC43} for $z$.
Finally we need to check the validity of the system of equations
\eqref{FSConds1} and \eqref{FSConds2}. 
\\
Concerning  $y$,
 the function $\lambda(\cdot,z_1,z_2)$ is well-defined for
$ (z_1,z_2) \in supp\; \Pi$, thanks
 to point \ref{lambdaWellDefined} of Lemma \ref{LemmaEta} and
by definition we have
 $\lambda(dt, z_1, z_2) \ll \rho^S_{dt}, \quad \forall(z_1,z_2)\in D$,
which is Condition \eqref{CondLambda0}.
\\
In order to prove that $y \in \Da$, which was defined in 
\eqref{ExpAddADom},
 it remains
 to prove Conditions \eqref{CondLambda1} and \eqref{CondLambda3} of Theorem \ref{T17}.
Let $t\in[0,T],\;(z_1,z_2)\in D/2$. By \eqref{lambda}, we have
$$\left|\lambda(t,z_1,z_2)\right|= \exp\left( \int_t^T \frac{d{\rm Re}(\eta(z_1,z_2,u))}{d\rho^S_u} \rho^S_{du}\right),$$
which implies, by item \ref{pt2LemmaEta} of Lemma \ref{LemmaEta}, that
\begin{equation}
\label{lambdaBounded}
|\lambda(t,z_1,z_2)| \leq \exp\left(c_1 \rho^S_T\right),
\end{equation}
which gives in particular  \eqref{CondLambda1}: in fact
$\int_{\C^2}d|\Pi|(z_1,z_2) |\lambda(t,z_1,z_2)|^2 \leq e^{2 c_1 \rho^S_T} |\Pi|(\C^2) < \infty$.

Finally, to conclude that $y\in \Da$, we need to show  \eqref{CondLambda3}.
By construction, $\lambda$ verifies equation \eqref{lambdaODE1}.
 Hence, by \eqref{lambdaODE1} and Cauchy-Schwarz
 we get
\begin{align}
\label{E4250}
\begin{split}
\Big(\int_0^T d\rho^S_t \Big|  \frac{d\lambda(t,z_1,z_2)}{d\rho^S_t} &+ \lambda(t,z_1,z_2) \frac{d\kappa_t(z_1,z_2)}{d\rho^S_t}\Big| \Big)^2 \\
&= \left(\int_0^T d\rho^S_t \left| \lambda(t,z_1,z_2)\right|\left|\frac{d\kappa_t(0,1)}{d\rho^S_t}\frac{d\rho_t(z_1, z_2,0,1)}{d\rho^S_t} \right|\right)^2 \\
&\leq   \int_0^T  \left| \lambda(t,z_1,z_2)\right|^2\left|\gamma_t(z_1,z_2) \right|^2 d\rho^S_t \int_0^T \left|\frac{d\kappa_t(0,1)}{d\rho^S_t}\right|^2 d\rho^S_t\\
&\leq  \left( I_1(z_1,z_2) + I_2(z_1,z_2)\right) \int_0^T \left|\frac{d\kappa_t(0,1)}{d\rho^S_t}\right|^2 d\rho^S_t ,
\end{split}
\end{align}
with
\begin{align}
\label{EI12}
\begin{split}
I_1(z_1,z_2) &:= c_2 \int_0^T \left| \lambda(t,z_1,z_2) \right|^2 d\rho^S_t \\
I_2(z_1,z_2) &:= -c_3 \int_0^T \left| \lambda(t,z_1,z_2) \right|^2  \frac{d {\rm Re}(\eta(z_1,z_2,t))}{d\rho_t} d\rho^S_t
\end{split}
\end{align}
where we have used  item \ref{pt3LemmaEta} of Lemma \ref{LemmaEta}.
Since $\lambda$ is uniformly bounded, see \eqref{lambdaBounded}, we have 
\begin{equation}
\label{I_1Bounded}
I_1(z_1,z_2)\leq c_2 \rho^S_T \exp\left(2c_1 \rho^S_T\right).
\end{equation}
On the other hand, 
\begin{align}
\label{I_2Bounded}
\begin{split}
I_2(z_1,z_2) &= -c_3 \int_0^T {\rm Re}(\eta(z_1,z_2,dt)) \exp \left( \int_t^T 2 {\rm Re}(\eta(z_1,z_2,du))\right) \\
&\leq  c_3 \displaystyle \sup_{y\in I_0+i\R^2} - \int_0^T {\rm Re}(\eta(y_1,y_2,dt))\exp \left( \int_t^T 2 {\rm Re}(\eta(y_1,y_2,du)) \right),
\end{split}
\end{align}
which is finite by item \ref{pt4LemmaEta} of Lemma \ref{LemmaEta}.
 Integrating \eqref{E4250} with respect to $\vert \Pi \vert$,
 taking into account
  the two uniform bounds in $(z_1,z_2)$, i.e.
 \eqref{I_1Bounded} and \eqref{I_2Bounded}, 
we can conclude to the validity of 
\eqref{CondLambda3}, so that  $y\in \Da$.


We show similarly that $\widetilde y:=y\times id\in \Da$. In fact, 
for $t\in[0,T]$ and $x,s>0$, we have
\begin{equation*}
\widetilde y(t,x,s) = \int_{\C^2} d \Pi(z_1, z_2) x^{z_1} s^{z_2+1} \lambda(t,z_1,z_2) \quad
= \int_{\C^2} d \widetilde\Pi(z_1, z_2) x^{z_1} s^{z_2} \widetilde \lambda(t,z_1,z_2),
\end{equation*}
where $\widetilde \lambda(t,z_1,z_2) = \lambda(t,z_1,z_2-1)$
and $\widetilde\Pi$ is the Borel complex measure defined by 
$$
\int_{\C^2} d \widetilde\Pi(z_1, z_2) \varphi(z_1,z_2) = \int_{\C^2} d \Pi(z_1, z_2) \varphi(z_1,z_2+1), 
$$ for every bounded measurable function $\varphi$. 
Hence, $ supp\; \widetilde{\Pi} =  supp\; \Pi + (0,1) $.
 By  \ref{R411Item1} and \ref{R411Item5} in Remark \ref{R411}, we have
 $(0,1)\in D/2$ and $ supp\; \Pi \subset D/2$.
 Then, by Remark \ref{R23}, $ supp\; \widetilde{\Pi}\subset D$, so that Assumption \ref{A_Pi} is verified for $\widetilde{\Pi}$. Moreover, by definition of $\widetilde \Pi$, the conditions \eqref{CondLambda0} and \eqref{CondLambda1} are 
fulfilled replacing $\Pi$ and $\lambda$ with $\widetilde \Pi$ and
 $\widetilde \lambda$. In order to conclude that $\widetilde y\in \Da$, we need to show 
\begin{equation} \label{ECondb}
A := \int_0^T d\rho^S_t \int_{\C^2}d|\widetilde\Pi|(z_1,z_2)  \left|  \dfrac{d\lambda(t,z_1,z_2-1)}{d\rho^S_t} + \lambda(t,z_1,z_2-1) \dfrac{d\kappa_t(z_1,z_2)}{d\rho^S_t}\right| < \infty,
\end{equation}
  which corresponds to Condition \eqref{CondLambda3}
 for  $\Pi$ and $\lambda$ replaced by $\widetilde \Pi$ and  $\widetilde \lambda$.
 Notice that
\begin{align*}
\begin{split}
A &=  \int_0^T d\rho^S_t \int_{\C^2}d|\Pi|(z_1,z_2)  \left|  \dfrac{d\lambda(t,z_1,z_2)}{d\rho^S_t} + \lambda(t,z_1,z_2) \dfrac{d\kappa_t(z_1,z_2+1)}{d\rho^S_t}\right| \\
&=  \int_0^T d\rho^S_t \int_{\C^2}d|\Pi|(z_1,z_2)  \left|  \dfrac{d\lambda(t,z_1,z_2)}{d\rho^S_t} + \lambda(t,z_1,z_2) \left(\dfrac{d\rho_t(z_1,z_2,0,1)}{d\rho^S_t} + 
\dfrac{d\kappa_t(z_1,z_2)}{d\rho^S_t} + \dfrac{d\kappa_t(0,1)}{d\rho^S_t}\right)\right| \\
&\leq  A_1 + A_2+A_3,
\end{split}
\end{align*}
where
\begin{eqnarray*}
A_1 &:=& \int_0^T d\rho^S_t \int_{\C^2}d|\Pi|(z_1,z_2)  \left|  \dfrac{d\lambda(t,z_1,z_2)}{d\rho^S_t} + \lambda(t,z_1,z_2)\dfrac{d\kappa_t(z_1,z_2)}{d\rho^S_t}\right|, \\
A_2 &:=& \int_0^T d\rho^S_t \int_{\C^2}d|\Pi|(z_1,z_2)  \left|\lambda(t,z_1,z_2) \dfrac{d\kappa_t(0,1)}{d\rho^S_t} \right|, \\
A_3 &:=&  \int_0^T d\rho^S_t \int_{\C^2}d|\Pi|(z_1,z_2)  \left|\lambda(t,z_1,z_2) \dfrac{d\rho_t(z_1,z_2,0,1)}{d\rho^S_t} \right|.
\end{eqnarray*}
The first term $A_1$ is finite, since we already proved that  $y\in \Da$
and so condition \eqref{CondLambda3} is fulfilled. Moreover
\begin{eqnarray*}
A_2 &\leq & \norm{\dfrac{d\kappa_t(0,1)}{d\rho^S_t}}_{\infty} \int_0^T d\rho^S_t \int_{\C^2}d|\Pi|(z_1,z_2)  \left|\lambda(t,z_1,z_2)\right|.
\end{eqnarray*} 

The right-hand side is finite, thanks to point \ref{setD3}
 of Assumption \ref{setD} and the fact that $\lambda$ is uniformly bounded.

Finally, by Cauchy-Schwarz and item \ref{pt3LemmaEta} of Lemma \ref{LemmaEta}, taking into 
account Notation \eqref{Egamma}, by similar arguments as \eqref{E4250}, 
we have
\begin{eqnarray*}
\left(A_3\right)^2 &\leq & |\Pi|(\C^2)\rho^S_T \int_{\C^2}d|\Pi|(z_1,z_2)\int_0^T d\rho^S_t   \left|\lambda(t,z_1,z_2)\right|^2 
 \vert \gamma_t(z_1,z_2) \vert^2 \\
&\leq & |\Pi|(\C^2)\rho^S_T \int_{\C^2} d|\Pi|(z_1,z_2) (I_1(z_1,z_2) + I_2(z_1,z_2)),
\end{eqnarray*} where
$I_1(z_1,z_2)$ and $I_2(z_1,z_2)$ have been defined in \eqref{EI12}.
We have already shown in \eqref{I_1Bounded} and \eqref{I_2Bounded} that $I_1$ and $I_2$ are bounded on $ supp\; \Pi$, hence $A_3 <\infty$. In conclusion,
 it follows indeed that $\widetilde y \in \Da$ and 
Hypothesis \ref{FSConditions1}  of Corollary \ref{FSConditions} 
is verified.
We define $(t,x,s) \mapsto z(t,x,s)$ so that
 $s^2 z(t,x,s) = \widetilde \a(y)(t,x,s)$. This gives
\begin{equation}
\label{zExpAdd}
z(t,x,s) = \int_{\C^2} d \Pi(z_1, z_2) x^{z_1} s^{z_2-1} \lambda(t,z_1,z_2) \gamma_t(z_1,z_2), \; \forall t\in[0,T],\; x,s>0,
\end{equation}
Lemma \ref{LemmaZInt} below shows that \eqref{ERC43} is fulfilled 
and so Hypothesis \ref{FSConditions2} of Corollary \ref{FSConditions} 
 is verified.  

\medskip

 We go on verifying Hypothesis  \ref{FSConditions3} of Corollary
\ref{FSConditions}, i.e. the validity of  \eqref{E450} and 
\eqref{E450bis}.
Condition
 \eqref{E450bis} is straightforward since 
$\lambda(T,\cdot,\cdot) = 1$.
 The second equality in \eqref{E450} takes place by definition of $z$.
The first equality  holds true
integrating  \eqref{lambdaODE} 
thanks to 
 \eqref{CondLambda3}.
This proves   \ref{FSConditions3} of Corollary  \ref{FSConditions}.

Finally Corollary \ref{FSConditions}
implies that
 $(Y,Z,O)$, is a solution of the BSDE \eqref{FSasBSDE}
provided we establish a lemma.
\end{proof}
\begin{lemma}
\label{LemmaZInt}
Let $z$ be as in \eqref{zExpAdd}, where $\lambda, \gamma$ have
been respectively defined in \eqref{lambda} and \eqref{Egamma}. We have 
$$ \E{\int_0^T \left|z (u, X_{u-}, S_{u-})\right|^2 S^2_{u-} 
  \rho^S_{du} } < \infty. $$
In particular \eqref{ERC43} is fulfilled.
\end{lemma}

\begin{proof}
First, let us show that
\begin{equation}
\label{eqCor37}
\int_{\C^2} d\Pi(z_1,z_2) \int_0^T |\lambda(t,z_1,z_2)|^2 \rho_{dt}(z_1,z_2) < \infty.
\end{equation} For this, we use points \ref{pt3LemmaEta} 
and \ref{pt4LemmaEta} of Lemma \ref{LemmaEta}, \eqref{lambdaBounded}
and \eqref{E4600} we get 
\begin{eqnarray*}
\int_0^T |\lambda(t,z_1,z_2)|^2 \rho_{dt}(z_1,z_2) & = & \int_0^T |\lambda(t,z_1,z_2)|^2 \frac{d\rho_{t}(z_1,z_2) }{d\rho_t^S}\rho_{dt}^S  \\
&\leq &  \int_0^T |\lambda(t,z_1,z_2)|^2 \left(c_2 - c_3 \frac{d {\rm Re}(\eta(y_1,y_2,t))}{d\rho^S_t}\right)\rho_{dt}^S  \\
&\leq & c_2 e^{2c_1 \rho_{T}^S} \rho_{T}^S -  
c_3  \int_0^T {\rm Re}(\eta(z_1,z_2,dt))\exp \left( \int_t^T 2 {\rm Re}(\eta(z_1,z_2,du)) \right) \\
&\leq & c_2 e^{2c_1 \rho_{T}^S} \rho_{T}^S +\\
&& c_3\sup_{(\xi_1,\xi_2) \in I_0+i\R^2} - 
\int_0^T {\rm Re}(\eta(\xi_1,\xi_2,dt))\exp \left( \int_t^T 2 {\rm Re}(\eta(\xi_1,\xi_2,du))
 \right).
\end{eqnarray*}
Hence \eqref{eqCor37} is fulfilled.

Using Cauchy-Schwarz inequality, Fubini theorem and point \ref{pt3LemmaEta} of Lemma \ref{LemmaEta}, we have
\begin{eqnarray*}
&&\E{\int_0^T \left|  z (u, X_{u-}, S_{u-})\right|^2 S^2_{u-} 
  d\rho^S_u }
=\E{\int_0^T \left|\int_{\C^2} d \Pi(z_1, z_2) X_{t-}^{z_1} S_{t-}^{z_2} \lambda(t,z_1,z_2) \gamma_t(z_1,z_2)\right|^2 \rho^S_{ds}}\\
&\leq &|\Pi|(\C^2) \sup_{t\in[0,T],(a,b)\in I_0} \E{X_{t}^{2a} S_{t}^{2b}} \int_{\C^2} d |\Pi|(z_1, z_2) \int_0^T \left|\lambda(t,z_1,z_2) \gamma_t(z_1,z_2)\right|^2 \rho^S_{ds} \\
&\leq &|\Pi|(\C^2) \sup_{t\in[0,T],(a,b)\in I_0} \E{X_{t}^{2a} S_{t}^{2b}} \int_{\C^2} d |\Pi|(z_1, z_2) \int_0^T \left|\lambda(t,z_1,z_2) \right|^2 \rho_{dt}(z_1,z_2).
\end{eqnarray*}
The right-hand side is finite, thanks to \eqref{eqCor37}.

\end{proof}

One can prove that the \textbf{weak} F-S decomposition in Proposition \ref{FSExpAdd} is actually a strong F-S decomposition in the sense of Definition \ref{FSDefinition}.

\begin{theorem} \label{TFS}
Under Assumption
 \ref{setD}, the random variable
$$h=\int_{\C^2} d \Pi(z_1, z_2) X_{T}^{z_1} S_{T}^{z_2}$$ 
admits an F-S decomposition \eqref{FSDec} where $h_0 = Y_0$
 and $(Y,Z,O)$ is given in Proposition \ref{FSExpAdd}.
Moreover, if $h$ is real-valued then the decomposition $(Y,Z,O)$ is 
real-valued and it is therefore the unique F-S decomposition.
\end{theorem}
\begin{remark}
\label{RGeneraliz}
This statement is a generalization of the results of  \cite{gor2013variance}
(and \cite{Hubalek2006}) to the case of hedging under basis risk. This yields
 a characterization of the hedging strategy in terms of Fourier-Laplace 
 transform
 and the moment generating function.
\end{remark}

\begin{proof}
Since $\Pi$ is a finite measure, then $h$ is square integrable.
Indeed by Cauchy-Schwarz
 \begin{equation}
 \label{Ehcarre}
\E{h^2} \leq  |\Pi|(\C^2) \int_{\C^2} \E{|X_T|^{2{\rm Re}(z_1)}|S_T|^{2{\rm Re}(z_2)}} d\Pi(z_1,z_2) 
 \leq \left(|\Pi|(\C^2)\right)^2 \sup_{(a,b)\in I} \E{|X_T|^{a}|S_T|^b},
\end{equation}
 where $I$ is a bounded subset of $\R^2$ defined in Assumption \ref{setD}.
 By item \ref{setD1} of Assumption \ref{setD}
and item \ref{propAddPt3} of Proposition \ref{propAdd}, previous quantity is finite.
 
By item \ref{R411Item4} of Remark \ref{R411} and
 by Remark \ref{RMS}, the real-valued F-S decomposition of any real valued
 square integrable $\cF_T$-measurable random variable is unique.

As a consequence, if $h$ is real-valued then its F-S decomposition 
is also real-valued. In fact, if $(Y_0,Z,O)$ is an F-S decomposition of $h$,
 then $(\overline Y_0,\overline Z,\overline O)$ is also an F-S of $\overline h$
by item \ref{RFS_3} of Remark \ref{RFS}.
 Thus, by subtraction,  $(\operatorname{Im}(Y_0), \operatorname{Im}(Z), \operatorname{Im}(O))$ is an F-S decomposition with real-valued triplet 
of the real-valued r.v.   $\operatorname{Im}(h)=0$.
 By uniqueness $\operatorname{Im}(Y_0)$, $\operatorname{Im}(Z)$ and $\operatorname{Im}(O)$ are null and the decomposition $(Y_0,Z,O)$ is real valued.

Now, let $(Y,Z,O)$ defined in Proposition \ref{FSExpAdd}. 
It remains to prove that $(Y_0,Z,O)$ is a strong (possibly complex) 
 F-S decomposition in the sense of Definition \ref{FSDefinition}.
For this we need to show items \ref{RBSDEFS_1},\ref{RBSDEFS_2},\ref{RBSDEFS_3} of Remark \ref{RBSDEFS}.
Item \ref{RBSDEFS_1} has been the object of \eqref{Ehcarre}.

We show below item \ref{RBSDEFS_2}
 i.e. $\E{\int_0^T | Z_u|^2d\langle M^S \rangle_u} < \infty$ and $\E{\left(\int_0^T |Z_u|d\| V^S \|_u\right)^2} < \infty$.
The first inequality is stated in Lemma \ref{LemmaZInt}.
In order to prove the second one, we 
recall that, by Corollary \ref{CCS},
$$
dV^S _t = S_{t-} \kappa_{dt}(0,1) = S_{t-} \frac{d\kappa_{t}(0,1)}{d\rho^S_t} \rho^S_{dt}.
$$
Consequently
\begin{eqnarray*}
\E{\left(\int_0^T |Z_u|d\| V^S \|_u\right)^2} &=& \E{\left(\int_0^T |Z_u| \left|\frac{d\kappa_{u}(0,1)}{d\rho^S_u}\right|S_{u-} \rho^S_{du}\right)^2} \\
&\leq & 
\int_0^T \left|\frac{d\kappa_{u}(0,1)}{d\rho^S_u}\right|^2 \rho^S_{du} \E{\int_0^T |Z_u|^2 S_{u-}^2 \rho^S_{du}}, \\
\end{eqnarray*}
which is finite since, by item \ref{R411Item3} of Remark \ref{R411} which says that $(0,1) \in \cD$, taking into account Lemma \ref{LemmaZInt}.

To end this proof, we need to show item \ref{RBSDEFS_3} of Remark \ref{RBSDEFS}. 
For this we use Proposition  \ref{ExpAddSqIntgMg} for which we need
to check conditions \ref{ExpAddSqIntgMg_1} and \ref{ExpAddSqIntgMg_2}.
By item \ref{R411Item5}  of Remark \ref{R411} 
 we have $I_0 \subset D/2$ which constitutes item \ref{ExpAddSqIntgMg_1}. Item \ref{ExpAddSqIntgMg_2}
is verified by condition \eqref{eqCor37} is verified.
 Hence Proposition \ref{ExpAddSqIntgMg} implies that 
$
t\mapsto y(t,X_t,S_t) - \int_0^t \a(y)(u,X_{u-},S_{u-}) \rho^S_{dt}
$ is a square integrable martingale.
\end{proof}

\subsection{Diffusion processes}
\label{S412}

We set $\cO = \R\times E$, where $E=\R$ or $]0,\infty[$.
In this Section we apply Corollary \ref{FSConditions} to the diffusion processes $(X,S)$
modeled  in Section \ref{E22} whose dynamics is given by \eqref{expleDiff}. 
We are interested in the F-S decomposition of $h = g(X_T,S_T)$.
We recall the assumption in that context.
\begin{assumption}\
\label{ADiff} \begin{itemize}
\item $b_X$, $b_S$, $\sigma_X$ and $\sigma_S$ are continuous
and  globally Lipschitz.
\item  $g: \cO \rightarrow \R$ is continuous.
\end{itemize}
\end{assumption}
We recall that $(X,S)$ solve the strong martingale problem related to $(\Da, \a, A)$
where $A_t = t$, $\Da =    \mathcal{C}^{1,2}([0,T[\times \cO) \cap 
\mathcal{C}^{1}([0,T]\times \cO)   $.
For a function $y \in \Da$, obviously  $\tilde y \in \Da$ and
the operators $\a$ and $\widetilde{\a}$
 are given by
\begin{eqnarray*}
\a(y) &=& \partial_t y + b_S \partial_s y + b_X \partial_x y
+  \frac{1}{2} \left\lbrace |\sigma_S|^2\partial_{ss} y + |\sigma_X|^2 \partial_{xx} y + 2 \langle \sigma_S, \sigma_X\rangle \partial_{sx} y  \right\rbrace, \\
\widetilde{\a}(y) &=& |\sigma_S|^2 \partial_s y + \langle \sigma_S, \sigma_X\rangle \partial_x y.
\end{eqnarray*}

Conditions \ref{FSConditions3}  of   Corollary \ref{FSConditions} translates into
\begin{eqnarray}
\label{pdeFSdiff}
b_S z &=& \partial_t y + b_S \partial_s y + b_X \partial_x y +  \frac{1}{2} \left\lbrace |\sigma_S|^2\partial_{ss} y + |\sigma_X|^2 \partial_{xx} y + 2 \langle \sigma_S, \sigma_X\rangle \partial_{sx} y  \right\rbrace,  \nonumber \\
y(T,.,.) &=& g(.,.), \\
|\sigma_S|^2 z  &=& |\sigma_S|^2 \partial_s y + \langle \sigma_S, \sigma_X\rangle \partial_x y. \nonumber 
\end{eqnarray}

If, moreover, 
 $\frac{1}{|\sigma_S|}$ is locally bounded,
 then we have  
\begin{equation}
\label{E496a}
\left\{
\begin{aligned}
\partial_t y + B \partial_x y + \frac{1}{2} \left( |\sigma_S|^2\partial_{ss} y + |\sigma_X|^2 \partial_{xx} y + 2 \langle \sigma_S, \sigma_X\rangle \partial_{sx} y  \right) =0, 
 \\
y(T,.,.) = g(.,.)
\end{aligned}
\right.
\end{equation}
and
\begin{equation}  \label{E496b}
 z = \partial_s y + \dfrac{\langle \sigma_S, \sigma_X\rangle}{|\sigma_S|^2} \partial_x y,
\end{equation}
where 
\begin{equation}
\label{EB}
B=b_X - b_S \dfrac{\langle \sigma_S, \sigma_X\rangle}{|\sigma_S|^2}.
\end{equation}
$z$ is then locally bounded since $\sigma_S, \sigma_X$ and $\frac{1}{\vert \sigma_S\vert}$ 
are locally bounded and because
$y \in \Da$.

\begin{proposition}
\label{WeakFSDiffusion}
We suppose the validity of Assumption \ref{ADiff} 
and that $\vert \sigma_S \vert$ is always strictly positive. 
If $(y,z)$ is a solution of the system
\eqref{E496a} and \eqref{E496b},
 such that $y \in \Da$, then $(Y,Z,O)$ is a solution of the BSDE \eqref{FSasBSDE}, where 
$$ Y_t = y(t, X_{t}, S_{t}), \quad
Z_t =d z(t, X_{t}, S_{t}),\quad
O_t = Y_t - Y_0 - \int_{0}^{t} Z_u d S_u.
$$
\end{proposition}
\begin{proof}
It follows from Corollary \ref{FSConditions} for which we need to check the conditions \ref{FSConditions1}, \ref{FSConditions2} and \ref{FSConditions3}. 
Indeed, since $y, \tilde y \in \Da$,  Condition \ref{FSConditions1} holds; since $z$ is locally bounded, by item \ref{RC33_2} of Remark \ref{RC33},
Condition \ref{FSConditions2} is fulfilled. 
Condition \ref{FSConditions3} has been the object of the 
considerations above the statement of the Proposition.
\end{proof}
The result above yields the weak F-S decomposition for $h$.
In order to show that $(Y_0, Z, O)$ constitutes a true  F-S decomposition, we
 need to make use of
Remark \ref{RBSDEFS}.
 First we introduce another assumption.

\begin{assumption}\
\label{AFSDiff}
Suppose that the process $(X,S)$ takes values in $\cO$ and the validity of the 
following conditions.
\begin{enumerate}[label=\roman*)]
\item $g\in C^1$ such that $g$, $\partial_x g$ and $\partial_s g$ have polynomial growth.

\item $B$ is globally Lipschitz. 

\item $\partial_x B$, $\partial_s B$, $\partial_x \sigma_X$, $\partial_s 
\sigma_X,$ $\partial_x \sigma_S$ and $\partial_s \sigma_S$ exist, are continuous and have polynomial growth.
\item $\sigma_S$ never vanishes.
\end{enumerate}
\end{assumption}

\begin{theorem}
\label{FSDiffusion}
Suppose that Assumptions \ref{ADiff} and \ref{AFSDiff} are fulfilled and 
suppose the existence of a function $y:[0,T]\times \cO \rightarrow \R$ such 
that
 \begin{equation}
 \label{solY}
 y \in C^0([0,T]\times \cO) \cap C^{1,2}([0,T[\times \cO) \text{ verifies the PDE } \eqref{E496a} \text{ and has polynomial growth.}
 \end{equation}
 Then the F-S decomposition \eqref{FSDec} 
of $h=g(X_T,S_T)$ is provided by $(h_0,Z,O)$ where,
$h_0 = Y_0$,
$$Y_t=y(t,X_t,S_t),\; Z_t=z(t,X_t,S_t), O_t = Y_t-Y_0-\int_0^t Z_u dS_u,$$
and $z:[0,T] \times \cO \rightarrow \R$ is given by \eqref{E496b}.
\end{theorem}

\begin{proof}

Let $y:[0,T]\times \cO \rightarrow \R$ verifying \eqref{solY} 
and $z$ defined by \eqref{E496b}.
In order to show that the triplet given in Proposition \ref{WeakFSDiffusion}
 yields a true F-S decomposition, we need to show items \ref{RBSDEFS_1}, \ref{RBSDEFS_2}, \ref{RBSDEFS_3} of Remark \ref{RBSDEFS}.

First notice that the random variable $g(X_T,S_T)$ is square integrable,
 because $g$ has polynomial growth and $X$ and $S$ admit all moments, see Remark \ref{Rallmoments}.
So \ref{RBSDEFS_1} is verified.

In view of verifying item \ref{RBSDEFS_2} of Remark \ref{RBSDEFS}
we recall that 
$$ \a(id) = b_S, \widetilde\a(id)  = \vert \sigma_S \vert^2, A_t \equiv t \text{ and } \; z = \partial_s y + \dfrac{\langle \sigma_S, \sigma_X\rangle}{|\sigma_S|^2} \partial_x y.$$
 Indeed, since  $y$ has polynomial growth, it is forced to be unique since 
\cite[Theorem 7.6, chapter 5]{Karatzas1991Brownian} implies that 
\begin{equation}
\label{StochRep}
y(t,x,s) = \E{g(X^{t,x,s}_T, S^{t,x,s}_T)},
\end{equation}
where $(\widetilde X=X^{t,x,s}, \widetilde S=S^{t,x,s})$ is a solution of 
 $$
 d \left(
\begin{smallmatrix}
  \widetilde X_r  \\
  \widetilde S_r
 \end{smallmatrix} \right) = \Sigma(r,\widetilde X_r, \widetilde S_r) d\widetilde W_r + 
\left(
\begin{smallmatrix}
  B(r, \widetilde X_r, \widetilde S_r) \\
  0
 \end{smallmatrix} \right)dr,
 $$
with $\widetilde X_t=x, \; \widetilde S_t=s$,
 where $\widetilde W=(\widetilde W^1,\widetilde W^2)$ is a standard two-dimensional Brownian motion,  
 and
 $$
 \Sigma = \left(
\begin{smallmatrix}
  \sigma_{X,1} & \sigma_{X,2}  \\
  \sigma_{S,1} & \sigma_{S,2}
 \end{smallmatrix} \right).
 $$
We recall that  $B$ has been defined in \eqref{EB}.

By \eqref{StochRep}, a straightforward adaptation of
 \cite[Theorem 5.5]{Friedman1975StochasticVol1} yields that
the partial derivatives 
 $\partial_x y$ and $\partial_s y$
exist and are continuous on $[0,T] \times \cO$ and they
 have polynomial growth.

Using \eqref{E496b}, we have 
$z b_S= b_S\partial_s y + b_X \partial_x y - 
B \partial_x y.$ Now, since $\partial_x y$ and $\partial_s y$ have polynomial growth, and by assumption
$b_S$, $b_X$ and $B$ have linear growth, we get that $z b_S$ has polynomial growth. This gives, by Remark \ref{Rallmoments},
$$\E{\left(\int_0^T \left|z b_S\right|(t,X_t, S_t) dt\right)^2} < \infty.$$
On the other hand, using \eqref{E496b} and 
 Cauchy-Schwarz, we have 
\begin{eqnarray*}
 |z \sigma_S| &=& ||\sigma_S|\partial_s y + 
\frac{\langle \sigma_X,\sigma_S\rangle}{|\sigma_S|} \partial_x y | \leq  |\sigma_S| |\partial_s y| + |\sigma_X||\partial_x y |.
\end{eqnarray*}
Since $\sigma_X$, $\sigma_S$ have linear growth and $\partial_x y$ and $\partial_s y$ have polynomial growth, 
we get that $z \sigma_S$ has polynomial growth, which implies, by Remark \ref{Rallmoments}, that
$ \E{\int_0^T \left|z \sigma_S\right|^2(t,X_t, S_t) dt} < \infty. $
Consequently, item \ref{RBSDEFS_2} of Remark \ref{RBSDEFS} is fulfilled. 

In order to show the last item \ref{RBSDEFS_3},
taking into account Remark \ref{RE22},
 we need to prove that 
\begin{eqnarray*}
u\mapsto M^Y_u&=&
\int_0^u \partial_x y(r,X_r,S_r) \left(\sigma_{X,1}(r,X_r,S_r) d W^1_r+\sigma_{X,2}(r,X_r,S_r) d W^2_r\right) \\
&+& \int_0^u \partial_s y(r,X_r,S_r) \left(\sigma_{S,1}(r,X_r,S_r) d W^1_r+\sigma_{S,2}(r,X_r,S_r) d W^2_r\right)
\end{eqnarray*}
is a square integrable martingale.  This is due to the fact that $\partial_x y$ and $\partial_s y$ have polynomial growth, and that $\sigma_X$ and $\sigma_S$ have linear growth, and Remark \ref{Rallmoments}, which implies that
$$
\E{\int_0^T \lbrace(\partial_x y(r,X_r,S_r))^2 |\sigma_X(r,X_r,S_r)|^2  +(\partial_s y(r,X_r,S_r))^2 |\sigma_S(r,X_r,S_r)|^2 \rbrace du} < \infty.
$$
This concludes the proof of Theorem \ref{FSDiffusion}.
\end{proof}
Below we show that, under  Assumptions \ref{ADiff} and \ref{AFSDiff},
 Condition \eqref{solY}
is not really restrictive.
\begin{proposition} \label{PFSDiffusion} 
We assume the validity of Assumptions \ref{ADiff} and 
\ref{AFSDiff}.

Moreover we suppose the validity of one of the three items below.
\begin{enumerate}[label=\arabic*)]
\item \label{RegularCase} We set $\cO = \R^2$. Suppose that the second 
(partial, with respect to $(x,s)$) derivatives of $B$,
 $\sigma_X,$ $\sigma_S$ and $g$ exist, are continuous and
 have polynomial growth.

\item \label{UEllip}  We set $\cO = \R^2$.
We suppose $B$, $\sigma_X$, $\sigma_S$ to be  bounded and
there exist $\lambda_1,\lambda_2>0$ such that
$$
\lambda_1 |\xi|^2 \leq (\xi_1,\xi_2)C(t,x,s)(\xi_1,\xi_2)^T \leq \lambda_2 |\xi|^2, \; \forall \xi=(\xi_1,\xi_2)\in \cO,
$$
where $C(t,x,s) = \left(
\begin{smallmatrix}
  |\sigma_X|^2(t,x,s) & \langle \sigma_X, \sigma_S\rangle(t,x,s)  \\
  \langle \sigma_X, \sigma_S\rangle(t,x,s) & |\sigma_S|^2(t,x,s)
 \end{smallmatrix} \right)$.

\item \label{BScase} (Black-Scholes case).
We suppose $\cO = ]0,+\infty[^2$.
 \begin{eqnarray*}
b_S(t,x,s) = s\hat b_S, && \sigma_{S}(t,x,s) = (s \hat \sigma_{S,1} ,\; s \hat\sigma_{S,2}), \\
 b_X(t,x,s) = x\hat b_X, && \sigma_{X}(t,x,s) = (x \hat \sigma_{X,1} ,\; x\hat \sigma_{X,2}),
\end{eqnarray*}where $\hat b_S$, $\hat b_X$, $\hat\sigma_{S,1}$,
 $\hat\sigma_{S,2}$,
 $\hat\sigma_{X,1}$ and $\hat\sigma_{X,2}$ are constants,
such that $\langle \hat\sigma_X, \hat\sigma_S\rangle < |\hat\sigma_X||\hat\sigma_S|$.
\end{enumerate}
We have the following results.
\begin{description}
\item{i)} There is a (unique) strict solution $y$ of \eqref{E496a}
 of class $C^{1,2}([0,T[ \times \cO) \cap C^{0}([0,T] \times \cO)$
  with polynomial growth.
\item{ii)}
 The F-S decomposition \eqref{FSDec}
 of $h = g(X_T,S_T)$ is provided by $(h_0,Z,O)$ where
$(Y,Z,O)$ fulfills
$$Y_t=y(t,X_t,S_t),\; Z_t=z(t,X_t,S_t) \text{ and } O_t = Y_t-Y_0-\int_0^t Z_u 
dS_u,$$
 where $z$ is given by \eqref{E496b}.
\end{description}
\end{proposition}
\begin{remark} \label{RT420}
We will show below that under the hypotheses of Proposition \ref{PFSDiffusion},
then conclusion i) holds, i.e. 
 there is a function
$y$ fulfilling \eqref{solY}. We observe that, by the proof of Theorem
\ref{FSDiffusion}, if such a $y$ exists then it admits the probabilistic representation
\eqref{StochRep} and so it is necessarily the unique 
  $C^{1,2}([0,T[ \times \cO) \cap C^{0}([0,T] \times \cO)$, with polynomial
 growth, 
 solution of \eqref{E496a}.
\end{remark}

\begin{proof}

We proceed to discussing the existence of $y$ mentioned in Remark \ref{RT420}.
So we distinguish now the mentioned three cases.

Suppose first item \ref{RegularCase}.
The function $y$ defined  by \eqref{StochRep} is a continuous function
by the fact that the flow $(\widetilde X, \widetilde S)$
is continuous in all variables and Remark \ref{Rallmoments},
taking into account Lebesgue dominated convergence theorem.
  \cite[Theorem 6.1]{Friedman1975StochasticVol1}, 
 states that $y$ belongs to $C^{1,2}([0,T] \times \cO)$,
 and it verifies the PDE \eqref{E496a}. 
 \cite[Theorem 5.5]{Friedman1975StochasticVol1} says 
 in particular that $y$ has polynomial growth. 
In that case  conclusion i) is established.

Under the assumption described in   item \ref{UEllip},
the conclusion i) can be obtained by simply adapting the proof 
of \cite[Theorem 12, p.25]{Friedman1983PDE}.
Indeed, according to \cite[Theorem 8, p.19]{Friedman1983PDE}
 there is a fundamental solution $\Gamma: \{(t_1,t_2), 0\leq t_1<t_2\leq T\}\times	\R^2\times \R^2 \rightarrow \R$ such that 
\begin{equation}
\label{EGamma}
\Gamma(t_1,t_2; \gamma, \xi) \leq \frac{1}{a_1(t_2-t_1)} \exp\left(-\frac{-|\gamma-\xi|^2}{a_1(t_2-t_1)} \right),
\end{equation} where $a_1$ is a positive constant.

Now, by \cite[Theorem 12, p.25]{Friedman1983PDE}, the function $y$ defined by
\begin{equation}
\label{SolByFond}
y(t,x,s) = \int_{\R^2} \Gamma(t,T; (x,s),(\xi_1,\xi_2)) 
g(\xi_1, \xi_2) d\xi_1d\xi_2,
\end{equation}
is a strict solution of \eqref{E496a}, in particular it belongs
to    $\mathcal{C}^{1,2}([0,T[\times \R^2) \cap 
\mathcal{C}^{0}([0,T]\times \R^2)$.

Since g has polynomial growth then there exist $a_2>0$, $p>1$ such that, $\forall x,s \in\R$,
\begin{equation}
\vert g(x,s) \vert \leq  a_2(1+|x|^p +|s|^p). \label{E4102}
\end{equation}
Thus, by \eqref{SolByFond}, \eqref{EGamma}  and  \eqref{E4102},
for $x,s\in\R$ and $0\leq t\leq T$, we have
\begin{equation*}
\vert y(t,x,s)\vert \leq \frac{a_2}{a_1(T-t)} \int_{\R^2} 
(1+|\xi_1|^p +|\xi_2|^p) \exp\left(-\frac{\vert x-\xi_1\vert^2 + \vert s-\xi_2\vert^2}{a_1 (T-t)} \right) d\xi_1d\xi_2.
\end{equation*} 
So there is a constant $C_1(p, T)>0$ such that
$
\vert y(t,x,s)\vert \leq C_1(p, T) \left( 1 + \E{\vert x+G_1 \vert^p+\vert x+G_2 \vert^p}\right),
$
where $G=(G_1,G_2)$ is a two dimensional centered Gaussian vector
 with covariance matrix equal to $\frac{a_1(T-t)}{2}$ times the identity matrix.
  Since $p>1$, then there is a constant $C_2(p, T)$ such that
 \begin{eqnarray*}
\vert y(t,x,s)\vert &\leq & C_2(p,T) \left( 1 + |x|^p + |s|^p + \E{\vert G_1 \vert^p+\vert G_2 \vert^p} \right) \\
&\leq & C_3(p,T) \left( 1 + |x|^p + |s|^p \right),
\end{eqnarray*} where $C_3(p,T)$ is another positive constant.
%
In conclusion the solution $y$ given by \eqref{SolByFond} 
has polynomial growth.

We discuss now  the Black-Scholes case \ref{BScase}
showing that, also in that case, there is $y$ such that
 \eqref{solY} is fulfilled.
First notice that the uniform ellipticity condition in \ref{UEllip}
 is not fulfilled for this dynamics, so we consider a logarithmic change of variable.
 For a function $y\in \Da$, we introduce the function $\hat y:[0,T] \times \R^2  \rightarrow \R$ 
defined by
$
\hat y(t, x, s) =  y(t, \log(x), \log(s)),\; \forall t\in[0,T], x,s>0.
$
By inspection we can show that $y$ is a solution of \eqref{E496a} 
if and only if $\hat y$ fulfills
\begin{eqnarray}\label{PDE-BS}
0&=& \partial_t \hat y + 
\left(\hat b_X-\hat b_S  \dfrac{\langle \hat\sigma_S,
 \hat\sigma_X\rangle}{|\hat\sigma_S|^2} - \frac{1}{2}|\hat\sigma_X|^2\right) \partial_x \hat y -
\frac{1}{2}|\hat \sigma_S|^2 \partial_s \hat y + \nonumber \\
&& +\frac{1}{2} \left( |\hat \sigma_S|^2\partial_{ss} \hat y+
 |\hat \sigma_X|^2 \partial_{xx} \hat y + 
 2 \langle \hat \sigma_S, \hat \sigma_X\rangle \partial_{sx} \hat y  \right), \\
\hat y(T,.,.) &=& \hat g(.,.), \nonumber
\end{eqnarray} where $\hat g(x,s)=g(e^x,e^s),\;\forall x,s \in\R$. 
Notice that the PDE problem \eqref{PDE-BS}  has constant coefficients 
and it verifies the uniform ellipticity condition in \ref{UEllip}.

Moreover, since $g$ has polynomial growth, then there exist $c>0,p>1$ such that $g(x,s) \leq c( 1+x^p+s^p),\; \forall x,s>0$ again. Hence 
$\hat g(x,s) \leq c( 1+e^{px}+e^{ps}),\; \forall x,s\in\R$. 

Again, by simple adaptation of the proof of 
 \cite[Theorem 12, p.25]{Friedman1983PDE}, 
 we observe that
 equation \eqref{PDE-BS} admits a solution $\hat y$ in $C^{1,2}([0,T[\times\R^2) \cap C^0([0,T]\times\R^2)$, such that $\hat y(t,x,s) \leq K( 1+e^{px}+e^{ps}),\; \forall x,s\in\R,$
where $K>0$.
This yields that $y$ has polynomial growth, since $y(t, x, s) = \hat y(t, \log(x), \log(s)),\; \forall t\in[0,T], x,s>0$, so
$y(t,x,s) \leq K( 1+x^{p}+s^{p}),\; \forall t\in[0,T], x,s>0.$
This concludes the proof of conclusion i).

Conclusion ii) is now a direct consequence of Theorem \ref{FSDiffusion}
together with condition i).
\end{proof}
\begin{remark} \label{RHulley}
The last item of Proposition \ref{PFSDiffusion} permits to recover the results already found in \cite{Hulley2008}, by replacing
\begin{eqnarray*}
\hat b_S = (\mu_S-r), &&  \hat \sigma_{S} = (\sigma_S ,\; 0), \\
\hat b_X = (\mu_U-r), && \hat \sigma_{X} = (\rho \sigma_U ,\; \sqrt{1-\rho^2} \sigma_U),
\end{eqnarray*}where $\mu_S$, $\mu_U$, $r$, $\sigma_S$ and $\sigma_U$ are constants.
\end{remark}

\begin{appendices}
\numberwithin{equation}{section}

\section{Proof of Proposition \ref{PropMarkov}}
\label{PropMarkovProof}
\setcounter{equation}{0}

\begin{proof}
Let $f\in E$ and set $\widetilde f(x) = \frac{f(x)}{1+x^2}, \; \forall x\in \R$.
Condition \eqref{derFlow} implies, by mean value theorem, that there exists a constant $c(t)$ such that
$
\E{\left| X_t^{0,x}- X_t^{0,y} \right|^2} \leq c(t) \left|x-y\right|^2, \; \forall x,y\in \R.
$
Then, by the Garsia-Rodemich-Rumsey criterion, see for instance \cite[Section 3]{BarlowYor1982}, there exists a r.v. $\Gamma_t$ such that $\E{\Gamma_t^2}<\infty$ and $\forall x,y \in\R$ 
\begin{equation}
\label{XHolder}
\left|  X_t^{0,x}- X_t^{0,y} \right| \leq \Gamma_t \left|x-y\right|^\alpha, \; \text{for}\; 0 < \alpha	<\frac{1}{2},
\end{equation}
possibly up to a modified version of the flow.

This implies in particular that for $x\in\R$
\begin{eqnarray*}
\frac{|X_t^{0,x}|^2}{1+x^2} & \leq & \frac{2}{1+x^2} \left( |X_t^{0,0}|^2 + |X_t^{0,x} - X_t^{0,0}|^2 \right) \\
& \leq & \frac{2}{1+x^2} \left( |X_t^{0,0}|^2 + \left|\Gamma_t\right|^2 |x|^{2\alpha} \right) \leq  2 \left( |X_t^{0,0}|^2  + \left|\Gamma_t\right|^2 \right).
\end{eqnarray*}
Hence
\begin{equation}
\label{XMoment2}
\sup_{x\in\R} \E{\frac{|X_t^{0,x}|^2}{1+x^2}} < \infty.
\end{equation}
Consequently,  for $x\in \R$, we have
$$ \frac{\left|P_tf(x)\right|}{1+x^2}  = \frac{\left|\E{f(X_t^{0,x})}\right|}{1+x^2}  \leq  \norm{f}_E \frac{1+\E{|X_t^{0,x}|^2}}{1+x^2} 
 \leq  \norm{f}_E \sup_{\xi \in\R} \frac{1+\E{|X_t^{0,\xi}|^2}}{1+\xi^2}.
$$
The right-hand side is finite, thanks to \eqref{XMoment2}, so that
\begin{equation}
\label{PtBounded}
\norm{P_tf}_E \leq \norm{f}_E \sup_{\xi \in\R} \frac{1+\E{|X_t^{0,\xi}|^2}}{1+\xi^2}.
\end{equation}
After we will have shown that 
$\widetilde {P_tf}$ is also uniformly continuous,  \eqref{PtBounded} 
will also imply that $P_tf  \in E$ and that
 $P_t$ is a bounded linear operator.

Therefore it remains to show that $\widetilde{ P_tf}$ is uniformly continuous. For this, let $x,y\in\R$. We have
\begin{equation}
\label{EI1I2}
\frac{P_tf(x)}{1+x^2}-\frac{P_tf(y)}{1+y^2} = \E{\frac{f(X^{0,x}_t)}{1+x^2}  -  \frac{f(X^{0,y}_t)}{1+y^2}} 
 =  \E{I_1+I_2},
\end{equation}
where 
\begin{eqnarray*}
I_1 &=& \left(\widetilde f(X^{0,x}_t) - \widetilde f(X^{0,y}_t) \right) \frac{1+(X^{0,x}_t)^2}{1+x^2} \\
I_2 &=& \widetilde f(X^{0,y}_t) \left(  \frac{1+(X^{0,x}_t)^2}{1+x^2} -  \frac{1+(X^{0,y}_t)^2}{1+y^2}\right).
\end{eqnarray*}

Let $\epsilon>0$. By uniform continuity of $\widetilde f$, there exists $\delta_1>0$ such that
\begin{equation}
\label{fTildeUCont}
\forall a,b\in \R,\; |a-b|\leq \delta_1 \Rightarrow \left| \widetilde f(a) - \widetilde f(b) \right| < \epsilon.
\end{equation}
Since $\displaystyle \lim_{M\to \infty} \E{|I_1| \1_{|\Gamma_t| \ge M}}=0$,
there exists $M_1>0$ such that
\begin{equation} \label{E216}
\E{|I_1| \1_{|\Gamma_t| \ge M_1}} < \epsilon.
\end{equation}
We fix $0 < \alpha < \frac{1}{2}$ and we choose
 $\delta_2 = \left(\frac{\delta_1}{M_1}\right)^{1/\alpha}$.
Taking into account \eqref{XHolder} and \eqref{fTildeUCont}, for $|x-y|<\delta_2$  we have 
\begin{eqnarray*}
\E{|I_1| \1_{|\Gamma_t|<M_1}} & \leq & \E{\frac{1+(X^{0,x}_t)^2}{1+x^2} \left(\widetilde f(X^{0,x}_t) - \widetilde f(X^{0,y}_t) \right) \1_{\left|  X_t^{0,x}- X_t^{0,y} \right|<\delta_1}} \nonumber\\
& < &  \sup_{\xi\in\R} \E{\frac{1 + |X_t^{0,\xi}|^2}{1+\xi^2}} \epsilon.
\end{eqnarray*} The right-hand side is finite thanks to \eqref{XMoment2}. Consequently, if $|x-y|<\delta_2$, then
\eqref{E216} implies that 
\begin{equation}
\label{I_1}
\E{|I_1|} <  A_1 \epsilon,
\end{equation} where $A_1= 1+ \sup_{\xi\in\R} \E{1 + \frac{|X_t^{0,\xi}|^2}{1+\xi^2}}$.

Concerning $I_2$, we define
\begin{equation}
\label{FI_2}
F(\omega, z) = \frac{1+|X_t^{0,z}(\omega)|^2}{1+z^2}, \omega \in \Omega, z \in \R.
\end{equation}
Since $z\mapsto F(\cdot, z)$ is differentiable in $L^2(\Omega)$, by mean 
value theorem we get
$$
\E{|I_2|} = |x-y| \E{\left|\widetilde f(X_t^{0,y}) 
\int_0^1 \partial_z F(\cdot, a x + (1-a) y) da \right|} \leq |x-y| \norm{f}_E \sup_{z}\E{\left|\partial_z F(\cdot, z)\right|}.
$$
It remains to estimate the previous supremum. We have for $z\in\R$
$$
\partial_z F(\cdot, z) = 2\frac{X^{0,z}_t \partial_z X^{0,z}_t}{1+z^2} -  2 z \frac{1+|X^{0,z}_t|^2}{(1+z^2)^2}.
$$
So by Cauchy-Schwarz we get
$$
\E{|\partial_z F(\cdot, z)|}  \leq  2\left( \frac{\E{|X^{0,z}_t|^2}}{1+z^2} \frac{\E{|\partial_z X^{0,z}_t|^2}}{1+z^2} \right)^{1/2} + 2 \frac{|z|}{1+z^2} \frac{1+\E{|X^{0,z}_t|^2}}{1+z^2} \leq  A_2,
$$
where
$
A_2 = 2\left( \displaystyle\sup_{z}\frac{\E{|X^{0,z}_t|^2}}{1+z^2} \displaystyle\sup_{z}\E{|\partial_z X^{0,z}_t|^2} \right)^{1/2}
+ \left( 1 + \displaystyle\sup_{z}\frac{\E{|X^{0,z}_t|^2}}{1+z^2} \right).
$

By \eqref{derFlow} and \eqref{XMoment2} $A_2$ is finite and we get
\begin{equation}
\label{I_2}
\E{|I_2|} \leq A_2 \norm{f}_E |x-y|.
\end{equation}

Combining inequalities \eqref{I_1} and \eqref{I_2}, 
\eqref{EI1I2} gives the existence of $\delta>0$ such that
\begin{equation*}
|x-y|<\delta \Rightarrow \left|\frac{P_tf(x)}{1+x^2}-\frac{P_tf(y)}{1+y^2}\right| < \epsilon,
\end{equation*} so that the function $x\mapsto \frac{P_tf(x)}{1+x^2}$ is uniformly continuous.

In conclusion we have proved that $P_tf\in E$. $P_t$ is a bounded linear operator follows as a consequence
of \eqref{PtBounded}.
\end{proof}

\section{Proof of Theorem \ref{thLevyGen}}
\label{appA_BSDE}

\setcounter{equation}{0}

We recall that the semigroup $P$ is here given by
$ P_t f(x) = \E{f(x + X_t)}, x \in \R, t \ge 0$ and
$X$ is a square integrable L\'evy process vanishing at zero.
The classical theory of semigroup for L\'evy processes defines the semigroup $P$ on the set $C_0$ of continuous functions vanishing at infinity, equipped with the sup-norm $\norm{u}_\infty = \sup_x{|u(x)|}$, cf. for example \cite[Theorem 31.5]{SatoBook}. On $C_0$, the semigroup  $P$ is strongly continuous, with norm $\norm{P}=1$, and its generator $L_0$ is given by 
\begin{eqnarray}
\label{L0App}
L_0f(x) = \int \left(f(x+y)-f(x)-yf^\prime(x)\1_{|y|<1}\right)\nu(dy), \ f \in C_0.
\end{eqnarray}
Moreover, \cite[Theorem 31.5]{SatoBook} shows that $C^2_0 \subset D(L_0)$, where $C^2_0$ is the set of functions $f\in C^2$ such that $f$, $f^\prime$ and $f^{''}$ vanish at infinity.

To prove Theorem \ref{thLevyGen} which concerns the infinitesimal generator of the semigroup $P$ defined on the set $E$ (cf. \eqref{setE}) related to a square integrable pure jump L\'evy process, we adapt  the classical theory. Since we consider a space $(E, \norm{.}_E)$, different from the classical one,
i.e.  $(C_0, \norm{.}_\infty)$, we need to show that $(P_t)$ is still a strongly continuous semigroup.

\begin{proposition}
\label{LevyPtStongCont}
Let $X$ be a square integrable L\'evy process, then the semigroup $(P_t) : E \rightarrow E$ is strongly continuous.
\end{proposition}

\begin{proof}
The idea of the proof is an adaptation of the proof in \cite[Theorem 31.5]{SatoBook}.

Let $f\in E$ and $\widetilde{f}$ defined by $\widetilde{f}(x)=\frac{f(x)}{1+x^2}, \; \forall x\in \R$. 
We evaluate, for $t>0$, $x \in \R$
$$
\dfrac{P_tf(x) - f(x)}{1+x^2} = \E{\widetilde{f}(x+X_t)-\widetilde{f}(x)} + 
\E{\widetilde{f}(x+X_t)\frac{X_t^2+2x X_t}{1+x^2} }.
$$
So 
\begin{equation} \label{EB11}
 \norm{P_tf - f}_E \leq \sup_{x\in \R} \left| \E{ \widetilde{f}(x+X_t)-\widetilde{f}(x) } \right| + \sup_{x\in \R} \left| \E{ \widetilde{f}(x+X_t)\frac{X_t^2+2x X_t}{1+x^2} } \right|.
\end{equation}
First, notice that
$$
\left| \E{\widetilde{f}(x+X_t)\frac{X_t^2+2x X_t}{1+x^2}} \right| \leq  \norm{f}_E \E{ \frac{X_t^2+2|x X_t|}{1+x^2} }  \leq  \norm{f}_E \left(\E{X_t^2} + \E{|X_t|} \right),
$$
hence $$\sup_{x\in \R} \left| \E{ \widetilde{f}(x+X_t)\frac{X_t^2+2x X_t}{1+x^2}} \right| \leq \norm{f}_E \left(\E{X_t^2} + \E{|X_t|} \right).$$

Since $X$ is a square integrable L\'evy process,  $\E{X_t^2} =c_2t + c_1^2t^2$
where $c_1, c_2$ were defined in \eqref{LevySqInt12}.
Hence, the right-hand side of the inequality above goes to zero as $t$ goes to zero.

Now we prove that the first     term
$\sup_{x\in \R}\left|\E{\widetilde{f}(x+X_t)-\widetilde{f}(x)}\right|$
 in the right-hand side of \eqref{EB11}  goes to zero as well.
Let $\epsilon>0$ be a fixed positive real.
Since $\widetilde{f}$ is uniformly continuous, then there is $\delta >0$ such that
$\forall x,y \; |x-y|<\delta \; \Rightarrow \; |\widetilde{f}(x)-\widetilde{f}(y)|<\frac{\epsilon}{2}.$
Moreover, since $X$ is continuous in probability $$\exists t_0>0, \; \text{such that} \; \forall t<t_0 ,\; \P(|X_t|>\delta) < \frac{\epsilon}{4 \norm{f}_E}.$$

For all $x\in\R$, $t<t_0$ we have
\begin{align*}
\begin{split}
\left| \E{\widetilde{f}(x+X_t)-\widetilde{f}(x)} \right|&\leq \E{\left|\widetilde{f}(x+X_t)-\widetilde{f}(x)\right|  \1_{\{|X_t|\leq\delta\}} } +  \E{\left|\widetilde{f}(x+X_t)-\widetilde{f}(x)\right|  \1_{\{|X_t|>\delta\}} }\\
&\leq  \frac{\epsilon}{2} + 2 \norm{f}_E \P(|X_t|>\delta) \leq  \epsilon.
\end{split}
\end{align*}
Since the inequality above is valid for every $x\in\R$, then 
$ \sup_{x\in \R}\left|\E{\widetilde{f}(x+X_t)-\widetilde{f}(x) }\right| \xrightarrow{t\to 0}  0.$
This concludes the proof that $P$ is a strongly continuous semigroup.
\end{proof}

\begin{remark} \label{RContraction}
 Notice that the semigroup $(P_t)$ is not a contraction. In fact, 
if $f\in E, t>0$, then
\begin{equation} 
\label{ENormPt}
 \norm{P_t f}_E = \sup_{x\in \R} \left| \E{\frac{f(x+X_t)}{1+x^2}}\right|.
\end{equation}
Let $f_0(x)=1+x^2$ and denote again $c_1 = \E{X_1}$ 
and $c_2 = {\rm Var} (X_1)$. Obviously $f_0 \in E$,
$\norm{f_0}_E=1$ and
\begin{equation}
\label{ENormPt1}
\norm{P_t f_0}_E = \sup_{x\in \R} \E{\frac{1+(x+X_t)^2}{1+x^2}} = 1+\sup_{x\geq 0}\frac{2x|c_1|t + c_2t + c_1^2t^2}{1+x^2} = 1 + \vert c_1 \vert t +c_1^2  t^2.
\end{equation}
Hence $(P_t)$ cannot be not a contraction since
$
\norm{P_t} \ge \norm{P_t f_0}_E > 1.
$
\end{remark}

\bigskip
On the other hand,  for $f\in E$, \eqref{ENormPt} gives
$ \norm{P_t f}_E 
 \leq  \norm{f}_E \norm{P_t f_0}_E. $
By  \eqref{ENormPt1} this implies that
$
\norm{P_t} \leq 1+(|c_1| + c_2)t + c_1^2t^2.
$
So, there exists a positive real $\omega>0$ such that
$
\norm{P_t} \leq e^{\omega t}.
$

Semigroups verifying the latter inequality are called 
{\bf quasi-contractions}, see \cite{PazyBook}. For instance, \cite[Corollary 3.8]{PazyBook} implies that 
\begin{equation}
\label{HilleYosidaQuasiContr}
\forall \lambda>\omega,\; \lambda I - L \; \text{is invertible}.
\end{equation}

At this point we show that the space $E^2_0$, defined in \eqref{SetE2_0},
 is a subset of $D(L)$ and that  formula \eqref{L0App} remains valid in 
$E^2_0$. This will be done adapting a technique described in 
\cite[Theorem 31.5]{SatoBook}, where it is stated  
 that $C^2_0$ is included in $D(L_0)$.  
The main tool  used for the proof of \cite[Theorem 31.5]{SatoBook}
 is the  small time asymptotics 
\begin{equation} \label{EAsymp}
\lim_{t \to 0} \frac{1}{t} \E{ g(X_t)} = \int g(x) \nu(dx),
\end{equation}
which holds for bounded continuous function $g$ vanishing on a
 neighborhood of the origin, see \cite[Corollary 8.9]{SatoBook}.
 This result has been extended to a class of unbounded functions by
 \cite[Theorem 1.1]{Figueroa2008small}. 
\eqref{EAsymp} is used in \cite[Proposition 2.3]{Figueroa2008small} to prove  that the quantity $\displaystyle\lim_{t \to 0} \dfrac{P_t g - g}{t}(x)$
 converges point-wise,  under some suitable conditions on the function $g$. 

We state a similar lemma below.
\begin{lemma}
\label{appLem1}
Let $f\in E^2_0$. For all   $x\in\R,$   the quantity 
\begin{equation}
\label{PointWiseCv}
\lim_{t \to 0} \dfrac{P_t f-f}{t}(x)
\end{equation} 
exists  and equals the right-hand side
of \eqref{L0App}.
\end{lemma}
\begin{remark}
\label{RappLem1}\
\begin{enumerate} [label=\arabic*)]
\item To be self-contained, we give below a simple proof of 
Lemma \ref{appLem1}, in the case when $X$ is a square integrable pure jump 
process.
\item
Later we will need to show that the point-wise convergence \eqref{PointWiseCv} holds according to the norm of $E$.
\end{enumerate}
\end{remark}
\begin{proof}
Let $f\in E^2_0$. First, we verify that the integral \begin{equation} \label{proofL}
\int\left(f(x+y)-f(x)-y f^\prime(x)\1_{|y|<1}\right)\nu(dy) 
\end{equation}is well-defined for all $x\in \R$, taking into account $\int y^2 \nu(dy) < \infty$ by \eqref{LevySqInt}.  

In fact, by Taylor expansion and since $f \in E^2_0$, then for every $x\in \R$ there exist $a, b\geq 0$ such that, for all $y\in \R$
\begin{eqnarray*}
|f(x+y)-f(x)|\1_{|y|\geq 1} &\leq& a (y^2 + 1) \1_{|y|\geq 1},\\
|f(x+y)-f(x)-f^\prime(x)y|\1_{|y|<1} &\leq& b y^2 \1_{|y|<1}.
\end{eqnarray*}

Let $t>0, x\in \R$. By Taylor expansion and Fubini theorem, recalling that 
$P_t f(x)= \E{f(x + X_t)}$ we have
$$
\dfrac{P_t f-f}{t}(x) = c_1 f^\prime(x) + \int_0^1 (1-a)\frac{1}{t} \E{f^{''}(aX_t+x)X_t^2} da.$$ 
 By abuse of notation, we denote by $L_0f(x)$ the integral \eqref{proofL}.
Taking into account
 \eqref{LevySqInt12} we have
\begin{eqnarray}
\label{GenLevyTaylor}
L_0f(x) &=& c_1 f^\prime(x) + \int\left(f(x+y)-f(x)-yf^\prime(x)\right)\nu(dy) \nonumber\\
&=& c_1 f^\prime(x) + \int_0^1 (1-a) \int_\R y^2 f^{''}(ay+x) \nu(dy) da.
\end{eqnarray}
Hence, it remains to show that ($x$ being fixed)
\begin{align}
\label{EB4bis}
\begin{split}
\dfrac{P_t f-f}{t}(x) - L_0f(x) &= \int_0^1 (1-a) \Big(\frac{1}{t}\E{X_t^2f^{''}(aX_t+x)} - \int_\R y^2 f^{''}(ay+x) \nu(dy)\Big) da \\
& \xrightarrow[t \to 0]\   0.
\end{split}
\end{align}
For $a \in [0,1]$, we denote 
$g(y) = y^2  f^{''}(ay+x)$. 
We have $g(y) \underset{y \to 0}{\sim} y^2   f^{''}(x).$ 
If $  f^{''}(x) \neq 0$, then 
 \cite[Theorem 1.1]{Figueroa2008small} (ii) implies that
 \begin{equation} \label{EFig}
\lim_{t \to 0}\frac{1}{t}\E{g(X_t)} = \int_\R g(y) \nu(dy).
\end{equation}
If $  f^{''}(x) = 0$, then  $g(y) = o(y^2)$ and
 \eqref{EFig} is still valid by 
\cite[Theorem 1.1]{Figueroa2008small} (i).
We conclude to the validity of \eqref{EB4bis}
by Lebesgue dominated convergence  theorem
taking into account that $ f^{''}  $ is bounded.
\end{proof}

As observed in a similar case in  \cite[Remark 2.4]{Figueroa2008small},
we will prove that the point-wise convergence proved in Lemma
 \ref{PointWiseCv} holds in the strong sense.  

For this purpose, we introduce  the linear subspace
\begin{equation*}
\widetilde{E} = \Big\{ f\in \cC \; \text{such that }\; \widetilde{f}:= x\mapsto \frac{f(x)}{1+x^2}\; \text{is vanishing at infinity}\;\Big\} 
\end{equation*}
of $E$. It it is easy to show that $\widetilde E$ is closed in $E$ so that it is a Banach subspace of $E$.

\begin{lemma}
\label{appLem2}
Let $f, \; g\in \widetilde{E}$, such that
\begin{equation}
\label{Ldiaz}
\lim_{t \to 0} \dfrac{P_t f-f}{t}(x)  = g(x), \; \forall x\in\R.
\end{equation}
Then $f\in D(L)$ and $ L f = g$.
\end{lemma}

\begin{proof}
We first introduce a restriction $\widetilde{P}$ of the semigroup $P$ to the
 linear subspace $\widetilde E$.
By Lebesgue dominated convergence theorem
and the fact that $\frac{1+ (X_t+x)^2}{1+ x^2} \le 2(\vert X_t \vert^2 + 1),$
 one can show that
 $P_t f \in \widetilde E$ for any $f \in \widetilde E$, $t \ge 0$.
Hence
 $(\widetilde P_t)$ is a semigroup on $\widetilde E$; 
 we denote by $\widetilde{L}$  its infinitesimal generator. 

As in \cite[Lemma 31.7]{SatoBook}, we denote by $L^\#f=g$, the operator defined by the equation \eqref{Ldiaz} for $f, g \in \widetilde E$
and by  $D(L^\#)$ its domain, i.e. the set of functions $f$
for which \eqref{Ldiaz} exists. 
 Then $L^\#$ is an extension of $\widetilde L$.

Fix $q>|c_1|+ c_2$. We prove first that
\begin{equation} \label{Eqlf} 
\forall f\in D(L^\#)   \quad   (qI - L^\#)f=0 \Rightarrow f=0.
\end{equation}

Let $f\in D(L^\#)$ such that $(qI - L^\#)f=0$. We denote $f^-=-(f\wedge 0)$ and $f^+=f \vee 0$. Suppose that $f^+ \neq 0$.
Since $\widetilde {f^+}$ is continuous and vanishing at infinity,
  there exists $x_1$ such that
$\frac{f^+(x_1)}{1+ x_1^2}=\displaystyle\max_{x} \dfrac{f^+(x)}{1+x^2} >0$.
Moreover $f(x_1) = f^+(x_1)$ .
Then
$$
\frac{\E{f(x_1+X_t)-f(x_1)}}{t} \leq  \frac{1}{t} \left( f(x_1)\frac{\E{1+(x_1+X_t)^2}}{1+x_1^2}- f(x_1)\right).
$$
Passing to the limit when $t \rightarrow 0$ it follows
$ L^\#f(x_1) \leq  f(x_1) (|c_1|+ c_2). $
Then
$(q-|c_1|- c_2)f(x_1) \leq 0, $
which contradicts the fact that $f(x_1)>0$. Hence, $f^+ = 0$. With similar arguments, we can show that $f^- = 0$ and so $f=0$,
which proves \eqref{Eqlf}.

By restriction, $(\widetilde P_t)$ fulfills $\Vert \widetilde P_t \Vert \le e^{\omega t}$, in particular it is
a quasi-contraction semigroup, so by 
 \eqref{HilleYosidaQuasiContr}, we can certainly choose $q>\max(|c_1|+ c_2,\omega)$, so that
 $qI-\widetilde{L}$ is invertible and  $R(qI-\widetilde{L})=\widetilde E$.\\
We observe that  $ D(\widetilde L) \subset D(L^\#)$. Let now
 $f\in D(L^\#)$; then $(qI-L^\#)f \in \widetilde E
=R(qI-\widetilde{L})$. Consequently, there is $v\in D(\widetilde{L})$ such that $(qI-L^\#)f = (qI-\widetilde{L})v$. So, $(qI-L^\#)(f-v)=0$.

By \eqref{Eqlf},  $(qI-L^\#)$ is injective,  so $f=v$ and $f\in D(\widetilde{L})$.
Consequently $\widetilde L f$ is given by $g$ defined in \eqref{Ldiaz}.
Finally, the fact that $D(\widetilde{L}) \subset D(L)$ and $\widetilde{L}$
is a restriction of $L$ allow to conclude the proof
of Lemma \ref{appLem2}.

\end{proof} 

We continue the proof of Theorem \ref{thLevyGen} making use  of
Lemmas \ref{appLem1} and Lemma \ref{appLem2}.\\
First, let us prove that $E^2_0 \subset \widetilde E$. Indeed by Taylor expansion, we have, for $f\in  E_0^2$
 $$\frac{f(x)}{1+x^2} =\frac{f(0)}{1+x^2} + \frac{x}{1+x^2} f^\prime(0) + \frac{x^2}{1+x^2} \int_0^1 (1-\alpha) f''(x\alpha) d\alpha.$$
 
 Since $\lim_{x\rightarrow\infty} f''(x\alpha)=0$ for all $\alpha\in ]0,1[$,
 then by Lebesgue theorem, we have that $\lim_{x\rightarrow\infty}  \frac{f(x)}{1+x^2}=0$, so $f\in \widetilde E$.
 
By Lemma \ref{appLem1}, it follows 
\begin{equation*}
\lim_{t \to 0} \dfrac{P_t f-f}{t}(x)  = L_0f(x), \; \forall x\in\R,
\end{equation*} where $L_0$ is given in \eqref{proofL}. In order to apply Lemma \ref{appLem2}, it remains to  show that 
$L_0f \in\widetilde E$. Using   relation \eqref{GenLevyTaylor}, for $x\in\R$ we get
$$
\frac{L_0f(x)}{1+x^2}= c_1 \frac{f^\prime(x)}{1+x^2} + \int_0^1 (1-a) \int_\R y^2 \frac{f^{''}(ay+x)}{1+x^2} \nu(dy) da.
$$
Since $f\in E^2_0$, then  $f''$ is bounded and $f^\prime$ has linear growth.
 So, the fact that $\int_\R y^2 \nu(dy) <\infty$  implies 
indeed $\lim_{x\to \infty}\frac{L_0f(x)}{1+x^2}=0$ and $L_0f\in \widetilde E$.

Finally, Lemma \ref{appLem2} implies that $E_0^2 \subset D(L)$ and
for $f\in E^2_0 $, 
$Lf$ is given by \eqref{L_Levy}.
 
\end{appendices}

\noindent {\bf ACKNOWLEDGEMENTS}: 
The authors are grateful to both referees, the associated editor and
the editor in chief for helping them to improve the first
version of the paper.
Financial support was provided  by the ANR Project  MASTERIE 2010 BLAN--0121--01.
The second named author  also benefited partially from the
support of the ``FMJH Program Gaspard Monge in optimization and operation
research'' (Project 2014-1607H).

\bibliographystyle{plainnatMod}
\bibliography{biblio}
\end{document}